\documentclass[10pt]{amsart}

\usepackage[utf8]{inputenc}
\usepackage[T2A]{fontenc}
\usepackage[english]{babel}
\usepackage{amssymb,amsmath}
\usepackage{mathtools}
\usepackage{amsthm}
\usepackage{tikz-cd}
\usepackage{geometry}
\newtheorem{theorem}{Theorem}[section]
\newtheorem{proposition}[theorem]{Proposition}
\newtheorem{lemma}[theorem]{Lemma}
\newtheorem{corollary}[theorem]{Corollary}

\theoremstyle{definition}
\newtheorem{definition}[theorem]{Definition}
\newtheorem{remark}[theorem]{Remark}

\newcommand{\Z}{\mathbb{Z}}
\newcommand{\Q}{\mathbb{Q}}
\newcommand{\V}{\mathbb{V}}
\newcommand{\Zhat}{\widehat{\mathbb{Z}}}
\newcommand{\Zhatu}{\widehat{\mathbb{Z}}^{\times}}

\DeclareMathOperator{\Aut}{Aut}
\DeclareMathOperator{\Gal}{Gal}

\title{Locally conjugate Galois sections}
\author{Wojciech Porowski}
\address{Research Institute for Mathematical Sciences, Kyoto University, Kyoto 606-8502, Japan}
\email{porowski@kurims.kyoto-u.ac.jp}  

\begin{document}
\begin{abstract}
We consider sections of the \'etale homotopy exact sequence of a hyperbolic curve over a number field. We prove that two sections whose restrictions to decomposition groups are conjugate on a set of valuations of density one are globally conjugate, which establishes the local-global principle for the conjugacy classes of sections. In fact, we obtain this result as a corollary of a more general property concerning sections of the \'etale homotopy exact sequence, so-called finite covering property, which we prove as our main result. 
\end{abstract}

\maketitle

\section{Introduction}\label{s:introduction}

Let $X$ be a geometrically connected variety over a number field $K$, write $G_K=\mathrm{Gal}(\overline{K}/K)$ for the absolute Galois group of $K$ and consider the \'etale homotopy exact sequence
\begin{equation}\label{s1:seq:etale}
1\to \Delta_X\to \Pi_X\to G_K\to 1,
\end{equation}
where $\Pi_X$ and $\Delta_X$ are the \'etale fundamental groups of $X$ and $X_{\overline{K}}$, respectively (see \cite[Exp.~IX, Thm.~6.1]{sga1}). We omit basepoints from the notation as they do not play any role in what follows. From the functoriality of the \'etale fundamental group every $K$-rational point of $X$ determines a conjugacy class of sections of sequence~\eqref{s1:seq:etale}. Assuming that $X$ is a hyperbolic curve, the Grothendieck section conjecture predicts that every section is either cuspidal or comes from a (unique) $K$-rational point of $X$, see \cite{stix2012rational} for a precise statement and a comprehensive discussion. In general this conjecture is widely open.

In this work we will be interested in a local-global principle for conjugacy classes of sections of sequence~\eqref{s1:seq:etale}.
Write $\V(K)$ for the set of nonarchimedean valuations of $K$ and for every $v\in \V(K)$ write $G_v\subset G_K$ for a decomposition group associated to $v$.
Let $s$ and $t$ be two sections of sequence~\eqref{s1:seq:etale}. For $v\in \V(K)$ we say that $s$ and $t$ are conjugate at $v$ if the restrictions of $s$ and $t$ to $G_v$ are conjugate.
For a nonempty subset $\Omega\subset \V(K)$ we say that $s$ and $t$ are conjugate on $\Omega$ if they are conjugate at every $v\in \Omega$.
One may wonder whether two sections that are conjugate on a "large" set of valuations are necessarily globally conjugate. A similar problem has been considered recently in \cite{saidi2022torsors} in the case of pro-solvable fundamental groups of affine curves. Here we can provide an answer in the following form. 
\begin{theorem}\label{s1:thm:global_conjugacy}
Assume that $X$ is either a smooth curve or a semiabelian variety and let $\Omega\subset \V(K)$ be a set of density one. Then, any two sections of sequence~\eqref{s1:seq:etale} that are conjugate on $\Omega$ are (globally) conjugate.
\end{theorem}

In other words, Theorem~\ref{s1:thm:global_conjugacy} gives a certain local-global principle for conjugacy classes of sections. The essential content of the above theorem is the case of a hyperbolic curve as in the remaining cases the theorem was known before by a classical result of Serre, see Remark~\ref{s2:rem:serre}.
 In fact, in our approach Theorem~\ref{s1:thm:global_conjugacy} will be proved as a corollary of a more general property of the \'etale homotopy sequence, so-called finite covering property, which we now briefly introduce. 

Let $s,s_1,\ldots s_n$ be sections of sequence~\eqref{s1:seq:etale} for some $n\ge 1$ and let $\Omega$ be a nonempty subset of $\V(K)$.
We say that sections $(s_i)_i$ \textit{cover} section $s$ on $\Omega$ if for every $v\in \Omega$ there exists $1\le i\le n$ such that $s$ and $s_i$ are conjugate at $v$.
We say that sequence~\eqref{s1:seq:etale} has \textit{finite covering property} if for every $\Omega\subset \V(K)$ of density one and every $n\ge 1$ the following condition is satisfied: whenever $s,s_1,\ldots, s_n$ are sections of sequence~\eqref{s1:seq:etale} such that sections $(s_i)_i$ cover section $s$ on $\Omega$ then there exists $1\le i\le n$ such that $s_i$ and $s$ are conjugate.
With these definitions, we can state our main result.
\begin{theorem}\label{s1:thm:covering}
Assume that $X$ is either a smooth curve or a semiabelian variety. Then the \'etale homotopy exact sequence has finite covering property.
\end{theorem}
Note that Theorem~\ref{s1:thm:global_conjugacy} follows directly from Theorem~\ref{s1:thm:covering} by taking $n=1$ in the definition of the finite covering property.

We now sketch the structure of the proof of Theorem~\ref{s1:thm:covering} as well as the structure of the present paper. The first step consists of establishing the abelian version of Theorem~\ref{s1:thm:covering}, i.e., replacing the geometric \'etale fundamental group with its topological abelianization. It translates the problem into a question about the Tate module of a semiabelian variety which we answer in Section~\ref{s:covering_property} using a method of Stoll, introduced in~\cite[\S3]{stoll2007finite}. This already implies the validity of Theorem~\ref{s1:thm:covering} for semiabelian varieties and reduces the problem to the case of a hyperbolic curve.
In Section~\ref{s:quasi-sections} we discuss some elementary properties of sections of short exact sequences that are used later. In Section~\ref{s:global_conjugacy_special_case} we show a certain weaker version of Theorem~\ref{s1:thm:covering} by an easier argument; this version will not be needed for the proof of the main result. In Section~\ref{s:pro-solvable} we extend the abelian case to the maximal geometrically pro-solvable quotient by an inductive argument, generalizing the results of \cite{saidi2022torsors}. Finally in Section~\ref{s:global_conjugacy} we deduce Theorem~\ref{s1:thm:covering} from the pro-solvable case.

The technique we use to pass from the abelian case to other quotients of the geometric fundamental group involves intersecting conjugacy classes of sections with open subgroups of $\Pi_X$ and analysing the "splitting" of this conjugacy class. It is at this point that the notion of the covering of a section becomes useful as it will be stable with respect to this operation. In particular, in our approach, even if one is mainly interested in Theorem~\ref{s1:thm:global_conjugacy} it is necessary to introduce a more general situation as in Theorem~\ref{s1:thm:covering}. 

It is an interesting problem whether one can relax the density assumption in Theorem~\ref{s1:thm:global_conjugacy} and require only that $\Omega$ is of positive density. The methods we use in the present work involve passing to various finite field extensions of $K$ hence cannot be directly applied to this more general situation as one might lose control of the density of a preimage of a set of valuations. 

\section{Finite covering property}\label{s:covering_property}
Let $K$ be a number field, write $\overline{K}$ for an algebraic closure of $K$ and $G_K=\Gal(\overline{K}/K)$ for the absolute Galois group of $K$. Moreover, write $\V$ for the set of nonarchemedean valuations of $\overline{K}$ and $\V(K)$ for the set of nonarchimedean valuations of $K$, thus we have a surjective map $\V\twoheadrightarrow \V(K)$.
For $v\in V$ write $G_v\subset G_K$ for the decomposition subgroup of $v$, the absolute Galois group $G_K$ acts on $\V$ and two valuations $v,v'\in \V$ lie in the same orbit if and only if they restrict to the same valuation on $K$.
In this case $G_v$ and $G_{v'}$ are conjugate subgroups of $G_K$. By abuse of notation we also write $G_v$ for $v\in \V(K)$, here $G_v$ is well defined up to conjugation.
This will not lead to any confusion when we consider cohomology of $G_K$-modules and the restriction to $G_v$ as the conjugation map induces the identity on the group cohomology of $G_K$, see~\cite[Prop~1.6.3]{cohomology_of_number_fields}.

Let $M$ be a topological $G_K$-module, for every $v\in \V(K)$ we have a localization map:
\[
loc_v\colon H^1(G_K,M) \to H^1(G_v,M),
\]
defined by restricting cohomology classes to $G_v$.
When $\Omega\subset \V(K)$ is a nonempty subset of valuations of $K$, we define $loc_{\Omega}$ to be the product of maps $loc_v$ for all $v\in \Omega$ 
\[
loc_{\Omega}\colon H^1(G_K,M) \to \prod_{v\in \Omega} H^1(G_v,M).
\]
Note that even when $\Omega$ has density one it may happen that the map $loc_{\Omega}$ is not injective, for an example see~\cite[Ch.~III, 4.7, Lem.~7]{serre1994cohomologie}.
However, when $M$ has trivial $G_K$-action (and $\Omega$ has density one) then the map $loc_{\Omega}$ is clearly injective as $G_K$ is topologically generated by subgroups $G_v$ for all $v\in \Omega$, thanks to Chebotariev's density theorem.

\begin{remark}\label{s2:rem:serre}
We have a classical result of Serre (see \cite{serre1964groupes} and \cite{serre1971groupes}) which provides a class of $G_K$-modules $M$ for which the localization map $loc_{\Omega}$ is injective for every $\Omega\subset \V(K)$ of density one.
Namely, pick a prime number $\ell$ and let $M$ be a finite free $\Z_{\ell}$-module which is unramified almost everywhere, i.e., for all but finitely many $v\in \V(K)$ the inertia subgroup $I_v\subset G_v$ acts trivially on $M$.
Assume furthermore that $M$ is pure of nonzero weight (in the sense of~\cite{deligne1971hodge1}), then the map $loc_{\Omega}$ is injective for every $\Omega$ of density one. In particular, this applies to the cyclotomic character $\Zhat(1)$ and the Tate module $T(A)$ of an abelian variety $A$ over $K$. Moreover, it also implies the validity of Theorem~\ref{s1:thm:global_conjugacy} in the case of semiabelian varieties and curves which are not hyperbolic.
\end{remark} 
 
In this section we will be interested in a certain generalization of the above mentioned injectivity property.
\begin{definition}
Let $M$ be a topological $G_K$-module, $\Omega$ be a nonempty subset of $\V(K)$, $n$ be a positive integer and $c, c_1, \ldots, c_n$ be cohomology classes in $H^1(G_K,M)$. We say that classes $(c_i)_i$ \textit{cover} class $c$ on $\Omega$ if for every $v\in \Omega$ there exists some $1\le i \le n$ such that 
\[
loc_v(c) = loc_v(c_i).
\]
Furthermore, we say that $G_K$-module $M$ has \textit{finite covering property} if for every $\Omega$ of density one and every integer $n\ge 1$ the following condition is satisfied: whenever classes $(c_i)_i$ cover class $c$ on $\Omega$ as above then there exists some $1\le i \le n$ such that $c=c_i$.
\end{definition}
Note that classes $(c_i)_i$ cover $c$ if and only if the difference classes $(c_i-c)_i$ cover the trivial class.
Thus in the definition of the finite covering property we could take as well $c=0$ and obtain the same notion. By considering $n=1$, it is clear that if $M$ has finite covering property then the map $loc_{\Omega}$ is injective for every $\Omega$ of density one.

\begin{remark}
It is not difficult to give an example of a topological $G_K$-module $M$ such that the map $loc_{\Omega}$ is injective for every $\Omega$ of density one but $M$ does not have finite covering property. Indeed, take $M=\Z/2\Z$ and consider three quadratic fields $\Q(\sqrt{2}),\Q(\sqrt{17})$ and $\Q(\sqrt{34})$. They determine three surjective homomorphisms $G_{\Q}\twoheadrightarrow \Z/2\Z$, hence three nonzero classes in $H^1(G_{\Q},M)$. Moreover, these classes cover the zero class on $\V(\Q)$; this follows from the fact that every prime number $p$ is split in at least one of these quadratic field extension of $\Q$. Hence $M$ does not have finite covering property. On the other hand, as $G_K$-action is trivial, all maps $loc_{\Omega}$ are injective whenever $\Omega$ has density one.
\end{remark}

Our main goal in this section is to provide a class of $G_K$-modules having finite covering property. The argument we use is a mild extension of an argument of Stoll, see \cite[\S~3]{stoll2007finite}.

Let $M$ be a finite free $\Zhat$-module and suppose we are given a direct sum decomposition
\begin{equation}\label{s2:decompostion_of_a_module}
M=\bigoplus_{i=1}^nM_i,
\end{equation}
into finitely many nonzero finite free $\Zhat$-modules $M_i$. Let $T\subset \Aut_{\Zhat}(M)$ be a closed subgroup, we say that $T$ is a group of generalized homotheties with respect to decomposition~\eqref{s2:decompostion_of_a_module} if every $\varphi\in T$ preserves all submodules $M_i$ and for every $1\le i \le n$ the induced homomorphism $\varphi|_{M_i}\colon M_i\to M_i$ is a homothety.
In this case for each $1\le i\le n$ we write $T_i$ for the image of $T$ in $\Aut_{\Zhat}(M_i)$, hence $T_i$ may be identified with a subgroup of $\Zhatu$ and we have a natural surjection $T\twoheadrightarrow T_i$.  

In general, when $M$ is a finite free $\Zhat$-module and $T\subset \Aut_{\Zhat}(M)$ is a closed subgroup, we say that $T$ is a group of generalized homotheties if there exists a decomposition~\eqref{s2:decompostion_of_a_module} such that $T$ is a group of generalized homotheties with respect to this decomposition. Note that every such $T$ is an abelian subgroup of $\Aut_{\Zhat}(M)$.

When $G$ is a closed subgroup of $\Zhatu$ we say that $G$ has finite coexponent if the quotient $\Zhatu / G$ has finite exponent.

\begin{lemma}\label{s2:lem:cohomology_of_finite_coexponents_subgroups}
For a natural number $r\ge 1$ consider $N=(\Q/\Z)^{\oplus r}$ with an action of $\Zhatu$ given by homotheties and let $G$ be a closed subgroup of $\Zhatu$ of finite coexponent. Then we have:
\begin{enumerate}
\item the group $H^0(G,N)$ is finite,
\item the group $H^1(G,N)$ has finite exponent.
\end{enumerate}
\end{lemma}
\begin{proof}
We easily reduce to the case $N=\Q/\Z$. Since $G$ has finite coexponent, we have $H =(\Zhatu)^d\subset G$ for some natural number $d\ge 1$. Clearly $H^0(G,N)$ is a subset of $H^0(H,N)$, moreover we have an exact sequence
\[
1\to H^1(G/H,H^0(H,N))\to H^1(G,N) \to H^1(H,N).
\]
It follows that if the statement holds for $H$ then it holds for $G$ as well. Therefore we may assume that $G=(\Zhatu)^d$ for some $d\ge 1$. In this case, the statement follows from the computation in the proof of \cite[Lem.~3.1]{stoll2007finite}.
\end{proof}

\begin{definition}
Let $T\subset \Aut_{\Zhat}(M)$ be a group of generalized homotheties with respect to decomposition~\eqref{s2:decompostion_of_a_module}.
We say that $T$ is \textit{full} if for every $1\le i\le n$ the group $T_i\subset \Zhatu$ has finite coexponent and the kernel of the homomorphism $T\twoheadrightarrow T_i$ has finite exponent.
\end{definition}

\begin{lemma}\label{s2:lem:cohomology_of_full_group_of_homotheties}
Let $M$ be a finite free $\Zhat$-module and let $T\subset \Aut_{\Zhat}(M)$ be a full group of generalized homotheties. Then we have:
\begin{enumerate}
\item the group $H^0(T,M\otimes_{\Zhat} \Q/\Z)$ is finite,
\item the group $H^1(T,M\otimes_{\Zhat} \Q/\Z)$ has finite exponent.
\end{enumerate}
\end{lemma}

\begin{proof}
By assumption, we have a decomposition
\[
M=\bigoplus_{i=1}^nM_i,
\]
into a finite direct sum of finite free $\Zhat$-modules such that $T$ is a full generalized homothety group with respect to this decomposition. Write $N=M\otimes_{\Zhat} \Q/\Z$ and $N_i=M_i\otimes_{\Zhat} \Q/\Z$ for $1\le i\le n$. For $j=0,1$ we have
\[
H^j(T,N)=\bigoplus_{i=1}^n H^j(T,N_i).
\]
Moreover $H^0(T,N_i)=H^0(T_i,N_i)$ and since groups $T_i$ have finite coexponent we obtain finiteness of $H^0(T_i,N_i)$ from Lemma~\ref{s2:lem:cohomology_of_finite_coexponents_subgroups}, thus $H^0(T,N)$ is finite as well.

It remains to show that for every $1\le i \le n$ the group $H^1(T,N_i)$ has finite exponent. Write $H_i$ for the kernel of $T\twoheadrightarrow T_i$ and consider the exact sequence
\[
0\to H^1(T_i,N_i)\to H^1(T,N_i)\to H^1(H_i,N_i).
\]
As $T$ is assumed to be full, group $H_i$ has finite exponent and the same is true for its cohomology. Since $H^1(T_i,N_i)$ has finite exponent by Lemma~\ref{s2:lem:cohomology_of_finite_coexponents_subgroups}, the proof is complete.
\end{proof}

\begin{definition}
Let $M$ a finite free $\Zhat$-module with a continuous $G_K$-action, write $G$ for the image of $G_K$ in $\Aut_{\Zhat}(M)$. We say that $M$ is \textit{admissible} if there exists a decomposition
\begin{equation}\label{eq:decomposition_of_admissible_module}
M=\bigoplus_{i=1}^nM_i,
\end{equation}
into a finite direct sum of finite free and $G_K$-stable $\Zhat$-modules and a full group of generalized homotheties $T\subset\Aut_{\Zhat}(M)$ with respect to this decomposition such that $T$ is contained in $G$.
\end{definition}

\begin{remark}\label{s2:rem:examples_of_admissible_modules}
Before we continue, let us give three examples of admissible modules to motivate the above definition (these are the only examples that will be used later).
\begin{enumerate}

\item Cyclotomic character $\Zhat(1)$ is admissible. This follows from the isomorphism $\Gal(\Q^{cyc}/\Q)\cong\Zhat^{\times}$, where $\Q^{cyc}$ is the maximal cyclotomic extension of $\Q$, since then for a number field $K$ the image of $G_K$ in $\Aut_{\Zhat}(\Zhat(1))$ is an open subgroup. This also implies that for every $r \ge 1$ the $G_K$-module $\Zhat(1)^{\oplus r}$ is admissible. Here the decomposition~\eqref{eq:decomposition_of_admissible_module} is trivial, i.e., we have $n=1$.

\item The Tate module $T(A)$ of an abelian variety $A$ over a number field $K$ is an admissible $G_K$-module. This is a theorem of Serre proved in \cite{serre1986groupes}.
More precisely, for a prime number $\ell$ write $G_{K,\ell}$ for the image of $G_K$ in $\Aut_{\Z_{\ell}}(T_{\ell}(A))$, where $T_{\ell}(A)$ is the $\ell$-adic Tate module of $A$.
It was proved by Bogomolov in \cite{bogomolov1980} that $G_{K,\ell}\cap \Z^{\times}_{\ell}$ is an open subgroup of $\Z^\times_{\ell}$ for every prime number $\ell$, thus of finite index $c(\ell)$.
Serre proves in \cite{serre1986groupes} that the sequence $c(\ell)$ is bounded as $\ell$ varies and, after possibly replacing $K$ by a finite field extension, the natural map $G_K\to \prod_{\ell} G_{K,\ell}$ is surjective.
These two facts imply that the intersection of the image of $G_K$ in $\Aut_{\Zhat}(T(A))$ with the group $\Zhatu$ of homotheties has finite coexponent, hence $T(A)$ is admissible.
Here again the decomposition~\eqref{eq:decomposition_of_admissible_module} is trivial. 

\item Direct sum $M=\Zhat(1)^{\oplus r}\oplus T(A)$ of $r\ge 1$ copies of the cyclotomic character with the Tate module of an abelian variety $A$ over $K$ is an admissible $G_K$-module.
To prove it, we take the decomposition $M_1=\Zhat(1)^{\oplus r}$ and $M_2=T(A)$. From the previous example, there exists a closed subgroup $H$ of $G_K$ whose image in $\Aut_{\Zhat}(M_2)$ is a subgroup of $\Zhatu$ of finite coexponent.
Write $G$ and $T$ for the images of $G_K$ and $H$ in $\Aut_{\Zhat}(M)$, respectively. We have $T\subset G$ and we claim that $T$ is a full group of generalized homotheties with respect to the decomposition $M=M_1\oplus M_2$.
Indeed, by construction $T$ is a group of generalized homotheties and $T_2$ is of finite coexponent. For $t\in T$ we write $t_1$, $t_2$ for its images in $T_1$ and $T_2$, respectively.
Recall that we have an isomorphism of $G_K$-modules $det(T(A))\cong \Zhat(g)$, where $g$ is the dimension of $A$, hence for every $t\in T$ we have $t_1^{g}=t_2^{2g}$. 
It follows that for $i=1,2$ the kernel of the projection $T\twoheadrightarrow T_i$ has finite exponent. Moreover, as $T_2$ has finite coexponent, there exists $d\ge 1$ such that $(\Zhatu)^{d}\subset T_2$. Hence
\[
(\Zhatu)^{2dg}\subset T_2^{2g} = T_1^{g}\subset T_1,
\]
which implies that $T_1$ also has finite coexponent. This proves the admissibility of $M$.
\end{enumerate}
\end{remark}

For our purposes, the most interesting property of admissible $G_K$-modules is contained in the following theorem.
\begin{theorem}\label{s2:thm:finite_covering_property}
Let $M$ be a finite free $\Zhat$-module with a continuous $G_K$-action. Suppose that $M$ is admissible and for every $v\in \V(K)$ we have $H^0(G_v,M)=0$. Then $M$ has finite covering property
\end{theorem}
The proof of Theorem~\ref{s2:thm:finite_covering_property} follows the same strategy as the proof of~\cite[Prop.~3.6]{stoll2007finite}. We start by proving two technical lemmas.  

Let $M$ be a finite free $\Zhat$-module with a continuous $G_K$-action. For a natural number $n\ge 1$ we write $M_n=M/nM$, thus each $M_n$ is a finite discrete $G_K$-module and $M=\varprojlim M_n$.
We further define $G$ and $G_n$ to be the image of $G_K$ in $\Aut_{\Zhat}(M)$ and $\Aut_{\Z/n\Z}(M_n)$, respectively. For every $n\ge 1$ we have $G_n=\Gal(K_n/K)$ for a finite Galois field extension $K_n/K$ and we have natural surjections $G_K\twoheadrightarrow G \twoheadrightarrow G_n$. 

\begin{lemma}\label{s2:lem:existence_of_m}[cf. \cite[Lem.~3.1]{stoll2007finite}]
Assume that $M$ is an admissible $G_K$-module. Then there exists a natural number $m\ge 1$ such that for every natural number $n\ge 1$ the group $H^1(G_n,M_n)$ is $m$-torsion.
\end{lemma}
\begin{proof}
From the definition of admissibility, there exists a decomposition 
\[
M=\bigoplus_{i=1}^n M_i
\]
and a full group of generalized homotheties $T\subset \Aut_{\Zhat}(M)$ with respect to this decomposition such that $T\subset G$. Write $N = M\otimes_{\Zhat} \Q/\Z$, from Lemma~\ref{s2:lem:cohomology_of_full_group_of_homotheties} we deduce that $H^0(T,N)$ is finite and $H^1(T,N)$ has finite exponent.

From the surjection $G \twoheadrightarrow G_n$ we have an injection $H^1(G_n,M_n)\hookrightarrow H^1(G, M_n)$ thus it is enough to find $m\ge 1$ which annihilates all groups $H^1(G, M_n)$ for $n\ge 1$. Since $M$ is $\Zhat$-free we have a short exact sequence
\[
0\to M_n \to N \xrightarrow[]{n} N \to 0.
\]
Taking cohomology we obtain an exact sequence
\[
0 \to H^0(G,N) \otimes \Z/n\Z \to H^1(G,M_n) \to H^1(G,N),
\]
and since $H^0(G,N)$ is a subset of a finite group $H^0(T,N)$, it is enough to prove that cohomology group $H^1(G,N)$ has finite exponent.

Note that $T\subset G$ is a central subgroup of $G$ hence we have an exact sequence
\[
0 \to H^1(G/T,H^0(T,N))\to H^1(G,N) \to H^1(T,N),
\]
thus the statement follows.
\end{proof}

For the purpose of proving Theorem~\ref{s2:thm:finite_covering_property} the admissibility assumption on $G_K$-module $M$ is only needed to apply Lemma~\ref{s2:lem:existence_of_m} and conclude that cohomology groups $H^1(G_n,M_n)$ have bounded exponent as $n$ varies over all natural numbers.
Thus from now on let us assume that there exists $m\ge 1$ which annihilates all groups $H^1(G_n,M_n)$ and we fix one such natural number $m$. 

For a finite extension $L'/L$ of number fields we write $Split(L'/L)\subset \V(L)$ for the set of all nonarchimedean valuations of $L$ which are split completely in $L'$.

\begin{lemma}\label{s2:lem:density_bound}[cf. \cite[Lem.~3.4]{stoll2007finite}]
Let $N\ge 1$ and $c\in H^1(G_K,M_N)$, define $n$ to be the order of $mc$ in the finite group $H^1(G_K,M_N)$. Write $S$ for the set of valuations $v\in \V(K)$ satisfying the following two conditions:
\begin{enumerate}
\item valuation $v$ splits completely in $K_N$,
\item the localized class $loc_v(c)$ in $H^1(G_v,M_N)$ is trivial.
\end{enumerate}
Then the set $S$ has upper density not greater than $\frac{1}{n[K_N:K]}$.
\end{lemma}

\begin{proof}
Let $S_0$ be a set of valuations of $K$ satisfying property (2) thus $S\subset S_0$ and $S=S_0\cap Split(K_N/K)$.
For a valuation $v\in \V(K)$ choose $w\in \V(K_N)$ lying over $v$ and consider the commutative diagram

\begin{equation}\label{diagram}
\begin{tikzcd}
0\arrow[r]& H^1(G_N,M_N)\arrow[r]& H^1(G_K,M_N)\arrow[r, "res"]\arrow[d, "loc_v"]& H^1(G_{K_N},M_N)\arrow[d, "loc_w"]\\
          &         & H^1(G_v,M_N)\arrow[r]& H^1(G_w,M_N).
\end{tikzcd}
\end{equation}

As $M_N$ has trivial $G_{K_N}$-action, $res(c)$ determines a homomorphism 
\[
\varphi\colon G_{K_N}\to M_N,
\]
which corresponds to a finite Galois field extension $L_N/K_N$ with 
\[
\Gal(L_N/K_N)\subset M_N.
\]
Denote by $n'$ the order of $res(c)$ in $H^1(G_{K_N}, M_N)$, by the definition of $m$ and exactness of the first horizontal row of the above diagram we see that $n'mc=0$.
Hence $n$ divides $n'$. Moreover, the image $\varphi(G_{K_N})\subset M_N$ is an abelian group of order $[L_N:K_N]$ and $n'$ is equal to the exponent of $\varphi(G_{K_N})$, which implies that $n'$ divides $[L_N:K_N]$. Together we conclude that $n$ divides $[L_N:K_N]$ and in particular $[L_N:K_N]\ge n$.

Let $f\colon \V(K_N)\twoheadrightarrow \V(K)$ be the natural map, note that for every subset $A\subset Split(K_N/K)$ we have 
\[
\delta(f^{-1}(A))=[K_N:K]\delta(A),
\]
where $\delta$ denotes the upper density of a set of valuations.
By the commutativity of diagram~\eqref{diagram} we have an inclusion 
\[
f^{-1}(S_0)\subset Split(L_N/K_N),
\]
thus after taking densities we obtain an inequality
\[
\delta(f^{-1}(S_0))\le \delta(Split(L_N/K_N)) = 1/[L_N:K_N]\le 1/n.
\]
Finally as $S\subset S_0$ we conclude
\[
[K_N:K]\delta(S)=\delta(f^{-1}(S))\le \delta(f^{-1}(S_0))\le 1/n,
\]
which finishes the proof.
\end{proof}

\begin{proof}[Proof of Theorem~\ref{s2:thm:finite_covering_property}]
Let $c_1,\ldots,c_n$ be finitely many classes in $H^1(G_K,M)$ covering the zero class on some set $\Omega\subset\V(K)$ of valuations of density one. Our goal is to prove that some class $c_i$ is trivial. By discarding some classes we may assume that for every $1\le i \le n$ the localization $loc_v(c_i)$ vanishes for at least one valuation $v\in\Omega$.

We first claim that some class $c_i$ is torsion. Suppose otherwise that all $c_i$ have infinite order in $H^1(G_K,M)$, hence the same is true for classes $mc_i$.
Because $H^1(G_K,M)=\varprojlim  H^1(G_K,M_i)$, there exists some natural number $N\ge 1$ such that the images of $mc_i$ in $H^1(G_K,M_N)$ have all order greater than $n$.
For $1\le i\le n$, let $S_i$ be the set of valuations $v\in V(K)$ satisfying the following two conditions:
\begin{enumerate}
\item valuation $v$ splits completely in $K_N/K$,
\item the image of $c_i$ in $H^1(G_v,M_N)$ is trivial.
\end{enumerate}
Then, it follows from Lemma~\ref{s2:lem:density_bound} that the upper density of $S_i$ is smaller than $1/(n[K_N:K])$, hence the sum $S=\bigcup_{i=1}^{n} S_i$ has upper density smaller that $1/[K_N:K]$.
As $\Omega$ has density one, this implies that there exists some valuation $v\in \Omega$ which splits completely in $K_N$ and does not belong to $S$.
But this is a contradiction with the assumption that classes $(c_i)_i$ cover the zero class on $\Omega$. Hence there exists some class $c_i$ which is torsion in $H^1(G_K,M)$. 

Finally we observe that a torsion class $c$ of $H^1(G_K,M)$ such that $loc_v(c)=0$ for at least one valuation $v$ must be trivial. Indeed, suppose that $dc=0$ for some $d\ge 1$ and consider the short exact sequence
\[
0\to M\xrightarrow{d} M\to M_d\to 0.
\]
Thanks to the assumption $H^0(G_v,M)=0$, from the long exact sequence in cohomology we obtain the following commutative diagram
\[
\begin{tikzcd}
0\arrow[r]& H^0(G_K,M_d)\arrow[r]\arrow[d, hook]&  H^1(G_K,M)\arrow[r, "d"]\arrow[d] & H^1(G_K,M) \\
0\arrow[r]& H^0(G_v,M_d)\arrow[r]&  H^1(G_v,M). &
\end{tikzcd}
\]
The exactness of the first row shows that $c=0$, which finishes the proof.
\end{proof}

In particular, Theorem~\ref{s2:thm:finite_covering_property} together with Remark~\ref{s2:rem:examples_of_admissible_modules} imply that if $T(A)$ is the Tate module of an abelian variety $A$ over $K$ then $G_K$-module $M=\Zhat(1)^{\oplus r}\oplus T(A)$ has finite covering property for every $r\ge 0$. This conclusion might be slightly strengthened by the following lemma.

\begin{lemma}\label{lem:extension_and_covering_property}
Let $M_1,M_2$ be two topological $G_K$-modules and assume that the direct sum $M_3=M_1\oplus M_2$ has finite covering property. Then both modules $M_1$ and $M_2$ also have finite covering property.
Assume furthermore that for every $v\in \V(K)$ we have $H^0(G_v,M_2)=0$. Then every extension $M$ of $G_K$-modules 
\begin{equation}\label{s2:eq:extension_of_modules}
0\to M_1 \to M \to M_2\to 0
\end{equation}
has finite covering property.
\end{lemma}
\begin{proof}
Fix a set of valuations $\Omega\subset \V(K)$ of density one, all covers below will mean covers on $\Omega$. First we prove that $M_1$ and $M_2$ have finite covering property, clearly it is enough to prove it only for $M_1$.
Let $c_1,\dots, c_n$ be classes in $H^1(G_K,M_1)$ covering the zero class. Consider classes $c'_1,\ldots,c'_n$ of $H^1(G_K,M_3)$ defined as $c'_i=(c_i,0)$ with respect to the identification
\[
H^1(G_K,M_3)= H^1(G_K,M_1)\oplus H^1(G_K,M_2).
\]
As $loc_v(c_i)=0$ implies $loc_v(c'_i)=0$, classes $(c'_i)_i$ cover the zero class in $H^1(G_K,M_3)$. By assumption for some $1\le i \le n$ we have $c'_i=0$ which implies $c_i=0$. This proves the finite covering property for $M_1$.

We now prove the second statement. Let $M$ be an extension~$\eqref{s2:eq:extension_of_modules}$ of $M_2$ by $M_1$ and let $c_1,\ldots, c_n$ be classes in $H^1(G_K,M)$ covering the zero class. Write $\bar{c}_1,\dots,\bar{c}_n$ for the images of $c_i$ along the homomorphism
\[
H^1(G_K,M)\to H^1(G_K,M_2).
\] 
Then, classes $(\bar{c}_i)_i$ cover the zero class in $H^1(G_K,M_2)$. Since $M_2$ has finite covering property, some classes $\bar{c}_i$ must be trivial; we may index them in such a way that $\bar{c}_i=0$ for $1\le i \le l$ and $\bar{c}_i\ne 0$ for $l+1\le i \le n$ with some $l\ge 1$. From our assumptions for every $v\in \V(K)$ we have a commutative diagram with exact rows
\[
\begin{tikzcd}
0\arrow[r]& H^1(G_K,M_1)\arrow[r]\arrow[d]& H^1(G_K,M)\arrow[d]\arrow[r] & H^1(G_K,M_2)\arrow[d] \\
0\arrow[r]& H^1(G_v,M_1)\arrow[r]& H^1(G_v,M)\arrow[r] & H^1(G_v, M_2),
\end{tikzcd}  
\]
thus for every $1\le i\le l$ there is a unique class $d_i$ in $H^1(G_K,M_1)$ mapping to $c_i$ through the second horizontal arrow in the first row. Note that $loc_v(c_i)=0$ if and only if $loc_v(d_i)=0$.

Consider classes $c'_i$ of $H^1(G_K,M_3)$ defined as follows
\[
c'_i = (d_i,0),\;\textrm{for}\; 1\le i\le l,\quad
c'_i = (0,\bar{c}_i),\; \textrm{for}\; l+1\le i\le n.
\] 
Observe that classes $(c'_i)_i$ cover the zero class in $H^1(G_K,M_3)$. Indeed, for every $v\in \Omega$ there is some $1\le i\le n$ such that $loc_v(c_i)=0$, since by assumption classes $(c_i)_i$ cover the zero class in $H^1(G_K,M)$.
If $i\le l$ then $loc_v(d_i)=0$ and if $i\ge l+1$ then clearly $loc_v(\bar{c}_i)=0$ so in both cases we have $loc_v(c'_i)=0$.

As $M_3$ has finite covering property, there exists $1\le i \le n$ such that $c'_i=0$. We cannot have $i\ge l+1$ since classes $\bar{c}_i$ are nonzero by construction, so $i\le l$ thus $d_i=0$ which implies that $c_i$ is trivial.
\end{proof}

\begin{corollary}\label{cor:extension_of_cyc_by_ab}
Let $M$ be a finite free $\Zhat$-module with a continuous $G_K$-action which is an extension of the Tate module $T(A)$ of an abelian variety $A$ over $K$ by a finite direct sum $\Zhat(1)^{\oplus r}$ of cyclotomic characters . Then $M$ has finite covering property.
\end{corollary}
\begin{proof}
This follows from Remark~\ref{s2:rem:examples_of_admissible_modules} and Lemma~\ref{lem:extension_and_covering_property}.
\end{proof}

We also need to discuss the relationship between the finite covering property of a module and the restriction of the action of $G_K$ to open subgroups.  

\begin{lemma}\label{lem:covering_property_and_restriction}
Let $M$ be a topological $G_K$-module and $G_L$ be an open subgroup of $G_K$. Denote by $M_L$ the $G_L$-module which is equal to $M$ as a $\Zhat$-module with the action restricted from $G_K$ to $G_L$. Assume that $G_L$-module $M_L$ has finite covering property and $H^0(G_L,M)=0$. Then $G_K$-module $M$ has finite covering property.
\end{lemma}

\begin{proof}
Let $\Omega\subset \V(K)$ be a subset of density one and define $\Omega'$ to be the preimage of $\Omega$ along the surjection $\V(L)\twoheadrightarrow \V(K)$, clearly $\Omega'\subset \V(L)$ is of density one. For $w\in \V(L)$ restricting to $v\in\V(K)$ we have a commutative diagram
\[
\begin{tikzcd}
H^1(G_K,M)\arrow[r, hook]\arrow[d]& H^1(G_L,M)\arrow[d] \\
H^1(G_v,M)\arrow[r]& H^1(G_w,M),
\end{tikzcd}
\]
here the injectivity of the upper horizontal arrow follows from our assumption. Let $c_1,\ldots,c_n$ be cohomology classes in $H^1(G_K,M)$ for some $n\ge 1$ which cover the zero class on $\Omega$.
For $1\le i\le n$ denote by $c_i'$ the image of $c_i$ in $H^1(G_L,M)$.
By the commutativity of the above diagram it follows that classes $(c'_i)_i$ cover the zero class on $\Omega'$.
From the finite covering property of $G_L$-module $M_L$ it follows that there exists some $1\le i\le n$ such that $c'_i=0$. Therefore $c_i=0$ which proves that $M$ has finite covering property. 
\end{proof}

For a topological group $G$ we write $[G,G]$ for the topological closure of the commutator subgroup of $G$ and $G^{ab}$ for the quotient $G/[G,G]$. As an application of the theory developed so far we obtain the following theorem, which is the main result of this section.

\begin{theorem}\label{s2:thm:abelianization_covering_property}
Let $X$ be either a semiabelian variety or a smooth geometrically connected curve over $K$, write $M=\pi_1^{et}(X_{\overline{K}})^{ab}$ for the (topological) abelianization of the \'etale fundamental group of $X_{\overline{K}}$. Then $G_K$-module $M$ has finite covering property.
\end{theorem}
\begin{proof}
There exists a finite field extension $L/K$ and an integer $r\ge 0$ such that either $M_L\cong \Zhat(1)^{\oplus r}$ or we have a short exact sequence  
\[
0\to \Zhat(1)^{\oplus r}\to M_L \to T(A)\to 0,
\]
where $T(A)$ is the Tate module of an abelian variety $A$ over $L$, see \cite[Rem.~1.3]{tamagawa1997grothendieck}. In particular, $H^0(U,M)=0$ for every open subgroup $U$ of $G_K$. Thus, by Remark~\ref{s2:rem:examples_of_admissible_modules} and Corollary~\ref{cor:extension_of_cyc_by_ab} we deduce that $M_L$ has finite covering property, hence Lemma~\ref{lem:covering_property_and_restriction} implies the same for $M$.
\end{proof}

\section{Quasi-sections}\label{s:quasi-sections}
In this section we collect some general facts about sections and quasi-sections of short exact sequences that will be used later.

Let $\Pi$ and $G$ be two profinite groups equipped with a fixed surjective homomorphism $pr\colon\Pi\twoheadrightarrow G$ called projection. Write $\Delta$ for the kernel of the projection so that we have a short exact sequence of profinite groups
\begin{equation}\label{s3:eq:fund_seq}
1\to \Delta\to \Pi\to G\to 1. 
\end{equation}

We will be interested in sections of this exact sequence, i.e., group homomorphisms $s\colon G\to \Pi$ such that the composition $G\to\Pi\twoheadrightarrow G$ is the identity.
We say that two sections $s$ and $t$ of sequence~\eqref{s3:eq:fund_seq} are conjugate (or, more precisely, $\Delta$-conjugate) if there exists $\delta\in\Delta$ such that $s(g)=\delta t(g)\delta^{-1}$ for all $g\in G$.
When sequence~\eqref{s3:eq:fund_seq} is split and we fix a section $s$ then we have the induced left action of $G_K$ on $\Delta$ defined as 
\[
g\delta = s(g)\delta s(g)^{-1},
\]
for all $g\in G$ and $\delta\in\Delta$. If $t$ is another section of sequence~\eqref{s3:eq:fund_seq} then $t$ may be expressed as 
\[
t(g)=a_gs(g),
\]
for all $g\in G_K$ and $a_g\in\Delta$. The map $g\mapsto a_g$ is a cocycle, i.e., it satisfies the formula 
\[
a_{gh}=a_g ga_h,
\]
for all $g,h\in G$. Cocycle $a_g$ determines a cohomology class in $H^1(G,\Delta)$ and this construction gives a bijection between conjugacy classes of sections and cohomology classes in $H^1(G,\Delta)$.

Let $U$ be an open subgroup of $\Pi$ and $s$ be a section of sequence~\eqref{s3:eq:fund_seq}. We say that $U$ is a neighbourhood of section $s$ when the image $s(G)$ is contained in $U$.
If $\Delta'$ is an open subgroup of $\Delta$ which is normalized by $s(G)$ then we may define an open subgroup $U=\Delta' s(G)$ of $\Pi$ and we have a short exact sequence
\[
1\to \Delta' \to U \to G \to 1.
\]
Clearly $U$ is a neighbourhood of section $s$ and it is easy to see that every neighbourhood of $s$ arises in this way. In particular, every characteristic open subgroup of $\Delta$ determines a neighbourhood of a section, we call it a characteristic neighbourhood of a section. 

To check whether two sections are conjugate we have the following well-known lemma.
\begin{lemma}\label{s3:lem:conjugacy_criterion}
Assume that $\Delta$ is topologically finitely generated and let $s$ and $t$ be two sections of sequence~\eqref{s3:eq:fund_seq}. Suppose that for every characteristic neighbourhood $U$ of $s$ there exists a $\Delta$-conjugate section $t'$ of $t$ such that $t'(G)$ is contained in $U$. Then $s$ and $t$ are $\Delta$-conjugate.
\end{lemma} 
\begin{proof}
As $\Delta$ is topologically finitely generated, there exists a descending series $(\Delta^i)_{i\ge 1}$ of open characteristic subgroups of $\Delta$ such that the intersection $\cap_{i\ge 1} \Delta^i$ is the trivial group. Write $U^i=\Delta^{i}s(G_K)$, it is a neighbourhood of section $s$.
For $i\ge 1$, write $S_i$ for the set of all $\bar{\delta}\in \Delta/\Delta^{i}$ such that $\delta t(G_K) \delta^{-1}$ is contained in $U^i$, where $\delta\in \Delta$ is a lift of $\bar{\delta}$. By assumption, for all $i\ge 1$ sets $S_i$ are nonempty and finite, thus the inverse limit $S=\varprojlim S_i$ is nonempty.
Take any $\delta\in S$ and observe that $\delta t(G_K)\delta^{-1}$ is contained in $\cap_{i\ge 1}U^i = s(G_K)$ thus $\delta t \delta^{-1} = s$.
\end{proof}

It will be convenient to extend slightly the definition of a section. Let $D$ be a closed subgroup of $\Pi$, we say that $D$ is a \textit{quasi-section} of sequence~\eqref{s3:eq:fund_seq} if $D$ maps isomorphically onto an open subgroup of $G$ through the projection.
Any section $s$ of sequence~\eqref{s3:eq:fund_seq} might be also considered as a quasi-section by taking $D=s(G)$ and clearly any quasi-section $D$ which surjects onto $G$ defines a section; we will use these two points of view on sections interchangeably.

For a closed subgroup $G'$ of $G$ write $\Pi'$ for the preimage of $G'$ along the projection, hence we have a short exact sequence
\begin{equation}\label{s3:eq:fund_seq_rest}
1 \to \Delta \to \Pi' \to G' \to 1. 
\end{equation}
This sequence will be called the restriction of sequence~\eqref{s3:eq:fund_seq} to $G'$ and the map $\Pi'\twoheadrightarrow G'$ is also called projection.
Any quasi-section $D$ of sequence~\eqref{s3:eq:fund_seq} induces a quasi-section $D|_{G'}=D\cap \Pi'$ of sequence~\eqref{s3:eq:fund_seq_rest} called the restriction of $D$ to $G'$.
When $G'$ is an open subgroup of $G$ then any quasi-section $D'$ of sequence~\eqref{s3:eq:fund_seq_rest} may also be regarded as a quasi-section of sequence~\eqref{s3:eq:fund_seq} simply by considering $D'$ as a subgroup of $\Pi$.
When necessary, we will specify whether we consider $D'$ as a quasi-section of sequence~\eqref{s3:eq:fund_seq} or \eqref{s3:eq:fund_seq_rest}.

For a quasi-section $D$ of sequence~\eqref{s3:eq:fund_seq} write $Z(\Pi,D)$ for the centralizer of the subgroup $D$ in $\Pi$ and define the geometric centralizer $Z(\Delta,D)$ of $D$ as the intersection $Z(\Pi,D)\cap\Delta$. Note that if $G$ is slim (i.e., all open subgroups of $G$ have trivial centre) then in fact $Z(\Pi,D)=Z(\Delta,D)$. In the following we will be mainly interested in sections which have trivial geometric centralizers.
We recall the following well-known criterion which guarantees the triviality of the geometric centralizer of a section. 
\begin{lemma}\label{s3:lem:centralizer_criterion}
Let $s$ be a section of sequence~\eqref{s3:eq:fund_seq} and consider $\Delta$ as a $G$-module with the action induced by $s$. Suppose that
\begin{enumerate}
\item group $\Delta$ is topologically finitely generated,
\item every open subgroup $\Delta'$ of $\Delta$ which is stabilized by $G$ satisfies 
\[
H^0(G,(\Delta')^{ab})=0.
\]
\end{enumerate}
Then section $s$ has trivial geometric centralizer.
\end{lemma}
\begin{proof}
Let $(\Delta^i)_{i\ge 1}$ be a descending sequence of open characteristic subgroups of $\Delta$ as in the proof of Lemma~\ref{s3:lem:conjugacy_criterion}. Let $\delta\in \Delta$ be an element centralizing section $s$.
Denote by $U^i$ the open subgroup of $\Delta$ generated by $\Delta^i$ and $\delta$. Observe that $U^i$ is stabilized by $G$ and $U^i/\Delta^i\subset H^0(G, (U^i)^{ab})$. By assumption we have $U^i= \Delta^i$ thus $\delta \in \Delta^i$ for every $i\ge 1$ hence $\delta = 1$. 
\end{proof}

From now on until the end of this section we will assume that: 
\begin{quote}
Every quasi-section of sequence~\eqref{s3:eq:fund_seq} has trivial geometric centralizer. 
\end{quote}

Our main application of the triviality of centralizers of quasi-sections is contained in the next lemma and its corollaries.
\begin{lemma}\label{s3:lem:equality_on_open_subgroup}
Let $s$ and $t$ be two sections of sequence~\eqref{s3:eq:fund_seq} such that $s(g)=t(g)$ for all $g$ in some open subgroup $H\subset G$. Then we have $s=t$.
\end{lemma}
\begin{proof}
Replacing $H$ by a smaller subgroup we may assume that $H$ is normal in $G$. Consider $\Delta$ as a $G$-module with the action induced by section $s$.
Write $t(g)=a_gs(g)$ for all $g\in G$ with some cocycle $a_g\in \Delta$. By assumption we have $a_h=1$ for all $h\in H$. For $g\in G$ and $h\in H$ we have 
\[
s(g)s(h)=s(gh)=a_{gh}t(gh)=a_{gh}t(g)t(h),
\]
thus equality $s(h)=t(h)$ implies that $a_{gh}=a_g$.
By normality of $H$ we also have $a_{hg}=a_g$ for all $g\in G$ and $h\in H$. This in turn implies, using the cocycle condition, that $ha_g=a_g$ for all $h\in H$ so $a_g$ belongs to the centralizer of $s(H)$.
Hence by assumption we have $a_g=1$ for all $g\in G$, so $s=t$.
\end{proof}
\begin{corollary}\label{s3:cor:conjugate_on_open_subgroup}
Let $s$ and $t$ be two sections of sequence~\eqref{s3:eq:fund_seq} and $G'$ be an open subgroup of $G$. Suppose that the restricted sections $s|_{G'}$ and $t|_{G'}$ are conjugate. Then $s$ and $t$ are conjugate. 
\end{corollary}
\begin{proof}
Replacing $t$ by a $\Delta$-conjugate section we may assume that $s(g)=t(g)$ for all $g\in G'$ and then from Lemma~\ref{s3:lem:equality_on_open_subgroup} we conclude $s=t$.
\end{proof}

Let $D_1$ and $D_2$ be two quasi-sections of sequence~\eqref{s3:eq:fund_seq}. We say that $D_1$ and $D_2$ coincide on an open subgroup if the intersection $D_1\cap D_2$ is open in $D_1$ (hence also in $D_2$).
We say that they are $\Delta$-conjugate if there exists $\delta\in \Delta$ such that $D_1=\delta D_2 \delta^{-1}$. Finally, we say that they are \textit{essentially} $\Delta$-conjugate if some $\Delta$-conjugate of $D_1$ coincides with $D_2$ on an open subgroup.
It is easy to say that the relation of being essentially $\Delta$-conjugate is an equivalence relation on the set of quasi-sections of sequence~\eqref{s3:eq:fund_seq} and it is in fact the smallest equivalence relation such that: (a) conjugate quasi-sections are equivalent; (b) a quasi-section and its restrictions  to open subgroups of $G$ are equivalent.

\begin{corollary}\label{s3:cor:coincide_on_open_subgroup}
Let $D$ and $E$ be two quasi-sections of sequence~\eqref{s3:eq:fund_seq} such that $D$ and $E$ coincide on an open subgroup. Suppose that we have an inclusion $pr(D)\subset pr(E)$. Then we have $D\subset E$.
\end{corollary}

\begin{proof}
Let $G'=pr(D)$ be the image of $D$ through the projection, thus $G'$ is an open subgroup of $G$. After restricting to $G'$ both $D|_{G'}$ and $E|_{G'}$ are sections of the restricted sequence~\eqref{s3:eq:fund_seq_rest} which coincide on an open subgroup. By Lemma \ref{s3:lem:equality_on_open_subgroup} we must have $D|_{G'}=E|_{G'}$ hence $D\subset E$.  
\end{proof}

For a quasi-section $D$ of sequence~\eqref{s3:eq:fund_seq} we define the set $(\Delta,D)$, called the conjugacy class of $D$, to be the collection of all quasi-sections of sequence~\eqref{s3:eq:fund_seq} which are $\Delta$-conjugate to $D$.
Note that $(\Delta,D)$ has natural left $\Delta$-action and in fact $(\Delta,D)$ is a $\Delta$-torsor as the geometric centralizer of $D$ is assumed to be trivial. Let $\Pi'$ be a closed subgroup of $\Pi$ thus we have a short exact sequence
\begin{equation}\label{s3:eq:ses_prime}
1\to \Delta'\to \Pi' \to G'\to 1,
\end{equation}
where $\Delta'$ and $G'$ are closed subgroups of $\Delta$ and $G$, respectively. For a quasi-section $D$ of sequence~\eqref{s3:eq:fund_seq} define $D'=\Pi'\cap D$, it is a quasi-section of sequence~\eqref{s3:eq:ses_prime} and is called the restriction of $D$ to $\Pi'$.
We write $\Pi'\cap (\Delta,D)$ for the set of subgroups of $\Pi'$ of the form
\[
\Pi'\cap \delta D \delta^{-1}, \;\textrm{for}\; \delta\in \Delta,
\]
we call this set the restriction of the conjugacy class $(\Delta, D)$ to $\Pi'$. It is clear that $\Pi'\cap (\Delta,D)$ has a left $\Delta'$-action (by conjugation), however it may no longer be a $\Delta'$-torsor.
We need to describe how $\Pi'\cap (\Delta, D)$ splits into a disjoint sum of $\Delta'$-torsors. It will be enough to consider the following two cases.

Case I. Suppose that $\Delta'=\Delta$ and $D'$ has trivial geometric centralizer. Then we have 
\[
\Pi'\cap \delta D \delta^{-1}= \delta(\Pi'\cap D)\delta^{-1} = \delta D' \delta^{-1},
\]
for every $\delta\in\Delta$. Hence $\Pi'\cap (\Delta, D) = (\Delta', D')$ and $(\Delta',D')$ is a $\Delta'$-torsor due to the triviality of the geometric centralizer of $D'$.
In other words, the restriction of the $\Delta$-torsor $(\Delta, D)$ to $\Pi'$ is still a $\Delta'$-torsor.

Case II. Suppose that $G'= G$ and $\Pi'$ is an open subgroup of $\Pi$. Then $\Delta'$ is of finite index $k=[\Delta:\Delta']$ in $\Delta$ and we may choose representatives $\delta_j\in\Delta$ for $1\le j\le k$ of right cosets of $\Delta'$ in $\Delta$ so that we have a disjoint sum decomposition
\[
\Delta=\bigsqcup_{j=1}^{k} \Delta'\delta_j.
\]
For each $1\le j \le k$ define quasi-sections $D_j$ of sequence~\eqref{s3:eq:fund_seq} by $D_j=\delta_j D\delta_j^{-1}$ and quasi-sections $D'_j$ of sequence~\eqref{s3:eq:ses_prime} by $D'_j=\Pi'\cap D_j$. Since $D'_j$ is open in $\Pi$ its geometric centralizer is trivial hence the conjugacy class $(\Delta',D'_j)$ is a $\Delta'$-torsor.
Moreover, observe that $\Delta'$-torsors $(\Delta',D'_j)$ for $1 \le j\le k$ are pairwise disjoint. Indeed, if $\delta'D'_i\delta'^{-1}=D'_j$ for some $1\le i,j\le k$ and $\delta'\in \Delta'$ then we have
\[
\Pi'\cap D_j = D'_j = \delta' D'_i\delta'^{-1} = \delta'(\Pi'\cap D_i)\delta'^{-1} = \Pi'\cap \delta'D_i\delta'^{-1}.
\] 
Therefore two conjugate quasi-sections $D_j$ and $\delta'D_i\delta'^{-1}$ coincide on an open subgroup, thus they are equal by Corollary \ref{s3:cor:coincide_on_open_subgroup}.
Furthermore, equality $D_j=\delta' D_i\delta'^{-1}$ implies that $\delta_j^{-1}\delta'\delta_i$ centralizes quasi-section $D$. Thus by the triviality of geometric centralizers we have $\delta'\delta_i=\delta_j$ hence $i=j$.

Finally, note that if $\delta=\delta'\delta_j$ for some $\delta'\in \Delta'$ then
\[
\Pi'\cap\delta D\delta^{-1} = \Pi'\cap\delta'\delta_j D\delta_j^{-1}\delta'^{-1} = \delta'(\Pi'\cap\delta_j D\delta_j^{-1})\delta'^{-1} = \delta' D'\delta^{-1}.
\]
Therefore every subgroup of $\Pi'$ contained in $\Pi'\cap (\Delta, D)$ is $\Delta'$-conjugate to $D'_j$ for some $1 \le j \le k$. Together with the disjointness of conjugacy classes $(\Delta',D'_j)$ it implies that we have a disjoint sum decomposition of $\Delta'$-sets
\[
\Pi'\cap (\Delta, D)= \bigsqcup_{j=1}^{k}(\Delta', D'_j).
\]
In other words, the restriction of a $\Delta$-torsor $(\Delta, D)$ to $\Pi'$ splits into a disjoint sum of $\Delta'$-torsors $(\Delta',D'_j)$ for $1 \le j \le k$. Clearly, this decomposition does not depend on the choice of a system of representatives $\delta_j$ of right cosets.

Note that when $\Pi'$ is an open subgroup of $\Pi$ then the two cases considered above completely describe the splitting of $\Pi'\cap (\Delta, D)$ into a disjoint sum of $\Delta'$-torsors by first restricting to an open subgroup $pr(\Pi')$ of $G$ and then to an open subgroup $\Delta'$ of $\Delta$.

We are mainly interested in situations when $G=G_K$ for a number field $K$, in this case quasi-sections of a short exact sequence
\begin{equation}\label{s3:ses:global}\tag{glob}
1\to \Delta\to \Pi \to G_K\to 1
\end{equation}
will be referred to as global quasi-sections. For every $v\in\V$ we may restrict sequence~\eqref{s3:ses:global} to $G_v$ and obtain another short exact sequence
\begin{equation}\label{s3:ses:local}\tag{loc}
1 \to \Delta \to \Pi_{v} \to G_v \to 1, 
\end{equation}
here $\Pi_v$ is the preimage of $G_v$ along the projection. Quasi-sections of sequence~\eqref{s3:ses:local}, for some $v\in \V$, will be referred to as local quasi-sections. For a global quasi-section $D$ we define a local quasi-section $D_v$ as the restriction $D_v=\Pi_{v}\cap D$, we also call $D_v$ the localization of $D$ at $v$.

For the rest of this section we assume that:
\begin{quote}
All local and global quasi-sections have trivial geometric centralizers. 
\end{quote}
In particular, the statement of Corollary~\ref{s3:cor:coincide_on_open_subgroup} holds for all local and global quasi-sections.

Let $D$ and $E$ be two global quasi-sections. For a valuation $v\in \V$ we say that $D$ and $E$ are $\Delta$-conjugate at $v$ if $D_v$ and $E_v$ are $\Delta$-conjugate local quasi-sections. We say that $D$ and $E$ are \textit{essentially} $\Delta$-conjugate at $v$ if some $\Delta$-conjugate of $E_v$ coincides on an open subgroup with $D_v$.

\begin{remark}
Note that if two global sections $s$ and $t$ are conjugate at $v\in\V$ then they are also conjugate at $v'=gv$ for every $g\in G_K$. Indeed, if there exists $\delta\in\Delta$ such that $s(x)=\delta t(x) \delta^{-1}$ for every $x\in G_v$ then for every $y\in gG_vg^{-1}=G_{gv}$ we have
\[
t(y)=t(g)\delta s(g^{-1})s(y)s(g)\delta^{-1} t(g^{-1})= \tilde{\delta}s(y)\tilde{\delta}^{-1},
\]
where $\tilde{\delta}=t(g)\delta s(g^{-1})\in \Delta$. Therefore in the case of global sections one can talk about conjugacy at $v\in\V(K)$ as it is the quotient of $\V$ by the action of $G_K$.
\end{remark}

We now introduce the notion of a covering of a global quasi-section which plays a crucial role in the theory. 

\begin{definition}
Let $D$ and $D_1,\ldots, D_n$ be global quasi-sections for some $n\ge 1$ and let $\Omega$ be a nonempty subset of $\V$. We say that quasi-sections $(D_i)_i$ \textit{cover} quasi-section $D$ on $\Omega$ if for every $v \in \Omega$ there exists $1\le i \le n$ such that $D$ and $D_i$ are essentially $\Delta$-conjugate at $v$.
\end{definition}

We would like to show that in certain sense the notion of a covering of a quasi-section is stable with respect to restricting to open subgroups of $\Pi$.
Let $D$ and $D_1,\ldots, D_n$ be quasi-sections of sequence~\eqref{s3:ses:global} such that $(D_i)_i$ cover $D$ on some nonempty $\Omega\subset \V$.
Let $\Pi'$ be an open subgroup of $\Pi$, thus we have an exact sequence 
\begin{equation}\label{s3:ses:global'}\tag{glob'}
1\to \Delta' \to \Pi' \to G'\to 1,
\end{equation}
with $\Delta'$ of finite index $k=[\Delta:\Delta']$ and $G'$ open in $G_K$.
As previously, we choose a system of representatives $\delta_j$ for $1 \le j\le k$ of right cosets of $\Delta'$ in $\Delta$. We define
\[
D_{ij}=\delta_j D_i\delta_j^{-1},\quad \textrm{for}\quad  1\le i\le n,\; 1\le j\le k,
\]
thus each $D_{ij}$ is a quasi-section of sequence~\eqref{s3:ses:global}. Write $D'$ and $D'_{ij}$ for the restrictions of $D$ and $D_{ij}$ to $\Pi'$, we regard $D'$ and $D'_{ij}$ as quasi-sections of sequence~\eqref{s3:ses:global'}. According to the previous discussion, for every $1\le i\le n$ we have a disjoint sum decomposition
\[
\Pi'\cap (\Delta, D_i)= \bigsqcup_{j=1}^{k}(\Delta', D'_{ij}).
\]

\begin{lemma}\label{s3:lem:stability_of_coverings}
With the above notation, quasi-sections $(D'_{ij})_{ij}$ cover quasi-section $D'$ on $\Omega$.
\end{lemma}

\begin{proof}
Let $v$ be a valuation in $\Omega$, by assumption there exists $1\le i\le n$ and such that $D_i$ and $D$ are essentially $\Delta$-conjugate at $v$. From the above decomposition we see that there exists $1\le j\le k$ such that $D'_{ij}$ and $D'$ are essentially $\Delta'$-conjugate at $v$. Thus, $D'$ is covered on $\Omega$ by quasi-sections $(D'_{ij})_{ij}$.
\end{proof}

We end this section by introducing the following definition.

\begin{definition}
We say that sequence~\eqref{s3:ses:global} has \textit{finite covering property} if for every $n\ge 1$ and every $\Omega\subset \V(K)$ of density one the following condition is satisfied: whenever $s$ and $s_1,\ldots s_n$ are sections of sequence~\eqref{s3:ses:global} such that sections $(s_i)_i$ cover $s$ on $\Omega$ then there exists $1\le i\le n$ such that $s$ and $s_i$ are $\Delta$-conjugate.
\end{definition}

Note that if we assume that $\Delta$ is abelian (so $\Delta$ has a natural $G_K$-action) then the finite covering property of sequence~\eqref{s3:ses:global} is equivalent to the finite covering property of $G_K$-module $\Delta$ introduced in Section~\ref{s:covering_property}.

\section{Covers of locally geometric sections}\label{s:global_conjugacy_special_case}
In this section we will apply results proved in Sections~\ref{s:covering_property} and~\ref{s:quasi-sections} to the homotopy exact sequence of \'etale fundamental groups of semi-abelian varieties and smooth curves over a number field.
More specifically, we will prove that this sequence has a certain weaker version of the finite covering property, where all sections are assumed to be locally geometric on a set of positive density.
Our main result, proved in Section~\ref{s:global_conjugacy}, does not logically depend on this weaker version, however the argument we present here is slightly shorter.

We keep the notation from Section~\ref{s:covering_property}. Let $X$ be a geometrically connected variety over a number field $K$. For a valuation $v\in \V$ we denote by $X_v=X\times K_v$ the base change of $X$ to the local field $K_v$, thus $X_v$ is a $K_v$-variety. We write $\Pi$ for the \'etale fundamental group $\pi^{et}_1(X)$ of $X$ and $\Delta$ for the geometric fundamental group $\pi^{et}_1(X_{\overline{K}})$, thus we have the \'etale homotopy short exact sequence
\begin{equation}\label{s4:ses:etale_global}\tag{\'et}
1 \to \Delta \to \Pi \to G_K \to 1.
\end{equation}
Applying the quotient  $\Delta\twoheadrightarrow \Delta^{ab}$ we obtain a short exact sequence
\begin{equation}\label{s4:ses:etale_abelianization}\tag{\'et-ab}
1\to \Delta^{ab}\to \Pi^{(ab)}\to G\to 1,
\end{equation}
where $\Pi^{(ab)}$ is the quotient of $\Pi$ by the normal subgroup $[\Delta,\Delta]$.
For a section $s$ of sequence~\eqref{s4:ses:etale_global} the composition of $s:G_K\to \Pi$ with the surjection $\Pi\twoheadrightarrow \Pi^{(ab)}$ is a section of sequence~\eqref{s4:ses:etale_abelianization} which we denote by $s^{ab}$ and call the abelianization of section $s$. 

For every valuation $v\in\V$ we have the decomposition subgroup $G_v\subset G_K$ and the restriction of sequence~\eqref{s4:ses:etale_global} to $G_v$ is denoted by
\begin{equation}\label{s4:ses:etale_local}
1 \to \Delta \to \Pi_v \to G_v \to 1. \tag{\'et-loc}
\end{equation}
Here we may identify $\Pi_v$ with the \'etale fundamental group $\pi^{et}_1(X_v)$ of $X_v$. Similarly, by passing to the quotient $\Delta\twoheadrightarrow \Delta^{ab}$ we obtain a short exact sequence
\begin{equation}\label{s4:ses:etale_local_abelianization}\tag{\'et-loc-ab}
1\to \Delta^{ab}\to \Pi^{(ab)}_v \to G_v\to 1,
\end{equation}
which is equal to the restriction of sequence~\eqref{s4:ses:etale_abelianization} to $G_v$.

\begin{remark}
It is known that when $X$ is either a semi-abelian variety or a smooth geometrically connected curve then the main assumption of Section~\ref{s:quasi-sections} is satisfied, namely that all quasi-sections of both sequences~\eqref{s4:ses:etale_global} and \eqref{s4:ses:etale_local} have trivial geometric centralizers.
Indeed, it is proved in \cite[Exp.~XII, Cor.~5.2]{sga1} that $\Delta$ is topologically finitely generated. Moreover, for an abelian variety $A$ over a $p$-adic local field $F$ the group of $F$-rational torsion points is finite, as the compact $p$-adic Lie group $A(F)$ has an open subgroup isomorphic to $\mathbb{Z}^{\oplus r}_p$ for some $r\ge 1$ (see \cite[Ch.~5, \S7]{serre1992lie}) which is torsion-free.
This implies that $H^0(G_F, T(A))=0$ and we conclude by applying Lemma~\ref{s3:lem:centralizer_criterion}, using the short exact sequence from the proof of Theorem~\ref{s2:thm:abelianization_covering_property}.
\end{remark}

The theory developed in Section~\ref{s:covering_property} implies the following result.

\begin{corollary}\label{s4:cor:finite_covering_property}
When $X$ is a semi-abelian variety then sequence~\eqref{s4:ses:etale_global} has finite covering property.
When $X$ is a hyperbolic curve then the abelianized sequence~\eqref{s4:ses:etale_abelianization} has finite covering property.
\end{corollary}
\begin{proof}
Since the geometric fundamental group of a semi-abelian variety is an abelian group, the statement follows from Theorem~\ref{s2:thm:abelianization_covering_property}.
\end{proof}

From now on we assume that $X$ is a hyperbolic curve. Write $S(\Pi_v,G_v)$ and $S(\Pi^{(ab)}_v,G_v)$ for the sets of conjugacy classes of sections of exact sequences~\eqref{s4:ses:etale_local} and \eqref{s4:ses:etale_local_abelianization}, respectively. We have a sequence of maps
\begin{equation}\label{s4:eq:kummer_map}
X(K_v)\to S(\Pi_v,G_v)\to S(\Pi^{(ab)}_v, G_v),
\end{equation}
whose composition is injective; here the first arrow is the Kummer map coming from functoriality of \'etale fundamental groups and the second one maps a section to its abelianization (see \cite{stix2012rational}, Ch. 7). When $s$ is a section of sequence~\eqref{s4:ses:etale_local} we say that $s$ is \textit{geometric} if its class in $S(\Pi_v,G_v)$ comes from a (unique) $K_v$-rational point of $X$.

Let $s$ be a global section of sequence \eqref{s4:ses:etale_global} and $v\in\V$. We say that $s$ is \textit{geometric at} $v$ if the localized section $s_v$ is geometric. Note that if $s$ is geometric at $v$ then it is also geometric at $gv$ for every $g\in G_K$ thus we may also talk about a section being geometric at $v\in \V(K)$. When $\Omega$ is a subset of $\V$, we say that $s$ is \textit{geometric on} $\Omega$ if for every $v$ in $\Omega$ section $s$ is geometric at $v$.

\begin{lemma}\label{s4:lem:covering_of_geometric_sections}
Let $s$ and $s_1,\ldots s_n$ for $n \ge 1$ be sections of sequence~\eqref{s4:ses:etale_global} such that sections $(s_i)_i$ cover section $s$ on a set $\Omega\subset \V(K)$ of density one.
Assume that there exists a nonempty set $T\subset \Omega$ such that sections $s$ and $s_i$ for all $1 \le i \le n$ are geometric on $T$. Then there exists $1\le i \le n$ such that $s$ and $s_i$ are conjugate at $v$ for every $v$ in $T$.
\end{lemma}
\begin{proof}
From Corollary~\ref{s4:cor:finite_covering_property} we conclude that there exists $1\le i\le n$ such that the abelianized sections $s^{ab}_i$ and $s^{ab}$ are conjugate. In particular, $(s_i^{ab})_v$ and  $s^{ab}_v$ are conjugate sections of sequence~\eqref{s4:ses:etale_local_abelianization} for every $v\in T$. Since they are assumed to be geometric, it follows from the injectivity of mapping~\eqref{s4:eq:kummer_map} that $(s_i)_v$ and $s_v$ are conjugate for every $v\in T$.
\end{proof}

We will also need a mild extension of Lemma~\ref{s3:lem:conjugacy_criterion}.

\begin{lemma}\label{s4:lem:conjugacy_of_sections}
Let $s$ and $s_1,\ldots, s_n$ be sections of sequence~\eqref{s4:ses:etale_global} for some $n\ge 1$ and let $M/K$ be a finite field extension. Suppose that for every neighbourhood $\Pi'$ of $s$ there exists some $1\le i \le n$ and a section $t_i$ of sequence~\eqref{s4:ses:etale_global} which is $\Delta$-conjugate to $s_i$ such that $t_i(G_M)$ is contained in $\Pi'$. Then, there exists $1\le i \le n$ such that $s$ and $s_i$ are conjugate. 
\end{lemma}
\begin{proof}
We first restrict sequence~\eqref{s4:ses:etale_global} to the open subgroup $G_M\subset G_K$. Then, using Lemma~\ref{s3:lem:conjugacy_criterion} we find an integer $1\le i\le n$ such that $s|_{G_M}$ and $s_i|_{G_M}$ are $\Delta$-conjugate.
From Corollary~\ref{s3:cor:conjugate_on_open_subgroup} we deduce that $s$ and $s_i$ are $\Delta$-conjugate.
\end{proof}

We are now ready to prove the following theorem, which gives a weak version of the finite covering property of the \'etale homotopy sequence~\eqref{s4:ses:etale_global} when all sections are assumed to be geometric on a fixed set of valuations of positive density. 

\begin{theorem}\label{s4:thm:global_conjugacy}
Let $s$ and $s_1,\ldots s_n$ for $n \ge 1$ be sections of sequence~\eqref{s4:ses:etale_global} such that sections $(s_i)_i$ cover section $s$ on a set $\Omega_K\subset \V(K)$ of density one. Assume furthermore that there exists a set $T_K\subset \V(K)$ of positive upper density such that sections $s$ and $s_i$ for all $1 \le i \le n$ are geometric on $T_K$. Then there exists $1\le i\le n$ such that $s$ and $s_i$ are conjugate.
\end{theorem}
\begin{proof}
Replacing $T_K$ with $\Omega_K\cap T_K$ we may assume that $T_K\subset \Omega_K$. Write $\Omega$ and $T$ for the preimages of $\Omega_K$ and $T_K$ through the surjection $\V\twoheadrightarrow \V(K)$, thus $T\subset \Omega \subset \V$. Let $\Pi'$ be a neighbourhood of $s$, thus we have a short exact sequence
\begin{equation}\label{s4:eq:ses_prime}
1 \to \Delta' \to \Pi' \to G_K \to 1,
\end{equation}
for some open subgroup $\Delta'\subset\Delta$. Thanks to Lemma \ref{s4:lem:conjugacy_of_sections}, it will be enough to find a section $t_i$ of sequence~\eqref{s4:ses:etale_global} which is $\Delta$-conjugate to $s_i$ for some $1\le i \le n$ such that $t_i(G_M)$ is contained in $\Pi'$, where $M$ is a finite field extension of $K$ which is independent of the neighbourhood $\Pi'$.

Denote $D=s(G_K)$ and $D_i = s_i(G_K)$ for $1 \le i \le n$, we also write $D'=D$ when $s$ is considered as a quasi-section of sequence~\eqref{s4:eq:ses_prime}. Recall from Section~\ref{s:quasi-sections} that for every $1 \le i \le n$ the $\Delta$-torsor $(\Delta, D_i)$ splits in $\Pi'$ as follows
\[
\Pi' \cap (\Delta, D_i) = \bigsqcup_{j=1}^{k}(\Delta', D'_{ij}).
\]
Here we have $D_{ij}=\delta_j D_i \delta_j^{-1}$ for a system of representatives $\delta_j$ of right cosets of $\Delta'$ in $\Delta$ and $D'_{ij}= \Pi' \cap D_{ij}$ is a restriction of $D_{ij}$ to $\Pi'$. 

Choose some finite field extension $L/K$ such that after restricting to $G_L$ all quasi-sections $D'_{ij}$ become sections; explicitly we may take $G_L$ to be the intersection of the projections of $D'_{ij}$ over all $i,j$, which is an open subgroup of $G_K$. Write $\Omega_L$ for the image of $\Omega$ under the surjection $\V\twoheadrightarrow \V(L)$, it is a subset of valuations of $L$ of density one.  
Write $D'_L=D'|_{G_L}$ and $D'_{ij,L}= D'_{ij}|_{G_L}$, hence $D'_L$ and $D'_{ij,L}$ are sections of the short exact sequence 
\[
1 \to \Delta' \to \Pi'_L \to G_L\to 1,
\]
where $\Pi'_L$ is the preimage of $G_L$ along the surjection $\Pi'\twoheadrightarrow G_K$.
From Lemma~\ref{s3:lem:stability_of_coverings} we see that sections $(D'_{ij,L})_{ij}$ cover section $D'_L$ over $\Omega_L$. Moreover, sections $D'_L$ and $D'_{ij,L}$ are geometric on $T$ and $\Omega_L$ is a density one subset of $\V(L)$.
Thus we may apply Lemma~\ref{s4:lem:covering_of_geometric_sections}, hence there exist some $i,j$ such that for all $v\in T$ sections $D'_L$ and $D'_{ij,L}$ are $\Delta'$-conjugate at $v$. Writing $E=D_{ij}$ we conclude that for all $v\in T$ two quasi-sections $D'_v$ and $E'_v$ of the short exact sequence
\begin{equation}\label{eq:ses_prime_local}
1\to \Delta' \to \Pi'_v \to G_v \to 1 
\end{equation}
are essentially $\Delta'$-conjugate. 

Fix $v\in T$, by the above there exists $\delta'\in \Delta'$ such that $\delta' E'_v\delta'^{-1}$ coincides on an open subgroup with $D'_v$. Define $F_v = \delta' E_v \delta'^{-1}$, it is a section of sequence~\eqref{eq:ses_prime_local} and note that as $\delta'\in \Delta'$ we have
\[
F'_v = F_v\cap \Pi' = \delta' E_v \delta'^{-1}\cap \Pi' = \delta' E'_v\delta'^{-1}.
\]
Hence two local sections $F_v$ and $D_v$ of sequence~\eqref{s4:ses:etale_local} coincide on an open subgroup, thus by Corollary~\ref{s3:cor:coincide_on_open_subgroup} we conclude that $F_v=D_v$, therefore $F'_v=D'_v$. Thus by taking projections we obtain
\[
pr(E'_v) = pr(F'_v) = pr(D'_v) = G_v,
\]
as $D'_v$ is a section of sequence~\eqref{eq:ses_prime_local}. Therefore for every $v\in T$ we have $pr(E'_v) = G_v$.

Consider now the image of $E'$ in $G_K$, we have $pr(E')=G_{K'}$ for a finite field extension $K'/K$. Equality $pr(E'_v)=G_v$ implies that for every $v\in T$ we have $G_v\subset G_{K'}$, thus every valuation $v\in T_K$ splits completely in $K'$. By Lemma~\ref{s4:lem:maximal_field} below there exists a number field $M$ which depends only on $T_K$ such that $K'\subset M$. Hence $E=D_{ij}$ determines a section $t_i$ of sequence~\eqref{s4:ses:etale_global} which is conjugate to $s_i$ and satisfies $t_i(G_M)\subset \Pi'$. Therefore we may apply Lemma~\ref{s4:lem:conjugacy_of_sections} which finishes the proof.
\end{proof}

\begin{lemma}\label{s4:lem:maximal_field}
Let $T_K\subset\V(K)$ be a set of valuations of positive upper density. Then there exists a finite field extension $M/K$ such that every valuation $v\in T_K$ splits completely in $M$ and $M$ is the maximal field extension with property.
\end{lemma}
\begin{proof}
We need to show that there are only finitely many finite field extensions $K'/K$ with the property that every $v\in T_K$ splits completely in $K'$. Enumerate them all as $K_i$ for $i\ge 1$, then by Chebotariev's density theorem we may bound the degree $[K_i\colon K]$ by $\alpha$, where $\alpha^{-1}$ is the upper density of $T_K$. Defining inductively $L_1=K_1$ and $L_{i+1}= L_iK_{i+1}$ for $i\ge 1$ we similarly have that $[L_i\colon K]$ is not greater than $\alpha$ hence there exists a number field $M$ such that $M=L_i$ for all sufficiently large integers $i$. 
\end{proof}

\section{Pro-solvable quotient}\label{s:pro-solvable}
In this section we consider the $m$-step solvable and the maximal geometrically pro-solvable variants of the \'etale homotopy sequence of a hyperbolic curve over a number field.
We prove that these sequences have finite covering property by induction, starting from the abelian case. This result will be used in Section~\ref{s:global_conjugacy} to pass from the maximal pro-solvable quotient to the full \'etale fundamental group.

When $G$ is a topological group we define inductively $G^{[m+1]}=[G^{[m]},G^{[m]}]$ for $m\ge 0$ where $G^{[0]}=G$, thus groups $G^{[m]}$ are closed characteristic subgroups of $G$.
For $m\ge 1$, denote by $G[m]$ the quotient $G/G^{[m]}$, note that $G[1]$ is just another notation for $G^{ab}$. Furthermore, write $G[sol]$ for the maximal pro-solvable quotient of $G$ whose kernel is equal to 
\[
G^{[sol]}=\bigcap_{m\ge 1}G^{[m]}.
\]
Hence for every $m$ we have surjections $G[sol]\twoheadrightarrow G[m]$ and $G[sol]$ is the inverse limit of groups $G[m]$ over all $m\ge 1$.

As in the previous section, consider a short exact sequence of profinite groups 
\begin{equation}\label{s5:ses:basic}\tag{\'etale}
1\to \Delta\to \Pi\to G_K\to 1,
\end{equation}
where $\Pi$ and $\Delta$ are the \'etale fundamental groups of a hyperbolic curve $X$ over $K$ and of its base change to $\overline{K}$.
For every $m\ge 1$ we also have the $m$-truncated short exact sequence 
\begin{equation}\label{s5:ses:m-truncated}\tag{$m$-step}
1\to \Delta[m] \to \Pi(m)\to G_K\to 1,
\end{equation}
as well as the maximal geometrically pro-solvable short exact sequence
\begin{equation}\label{s5:ses:solvable}\tag{pro-sol}
1\to \Delta[sol]\to \Pi(sol)\to G_K\to 1,
\end{equation}
where $\Pi(m)$ and $\Pi(sol)$ are quotients of $\Pi$ making the above sequences exact. In particular, we have quotients $\Pi\twoheadrightarrow \Pi(sol) \twoheadrightarrow \Pi(m)$.
The case $m=1$ of sequence~\eqref{s5:ses:m-truncated}, namely the maximal geometrically abelian quotient, has already been introduced in Section~\ref{s:global_conjugacy_special_case}.

The main goal of this section is to prove the following theorem.
\begin{theorem}\label{s5:thm:m-truncated_section_covering}
Short exact sequence \eqref{s5:ses:m-truncated} has finite covering property, for every $m\ge 1$.
\end{theorem}

We quickly note two corollaries of this theorem.

\begin{corollary}\label{s5:cor:solvable_covering_property}
Short exact sequence~\eqref{s5:ses:solvable} has finite covering property.
\end{corollary}
\begin{proof}
For a section $t$ of sequence~\eqref{s5:ses:solvable} and $m\ge 1$ denote by $t[m]$ the section of sequence~\eqref{s5:ses:m-truncated} obtained by composing $t$ with the quotient map $\Pi(sol)\twoheadrightarrow \Pi(m)$.
Let $\Omega_K\subset \V(K)$ be a subset of density one and let $s,s_1,\ldots,s_n$ be sections of sequence~\eqref{s5:ses:solvable}, for some $n\ge 1$, such that sections $(s_i)_i$ cover section $s$ on $\Omega_K$.
Clearly, sections $s_i[m]$ cover $s[m]$ on $\Omega_K$ for all $m\ge 1$. By Theorem~\ref{s5:thm:m-truncated_section_covering}, it follows that for every $m\ge 1$ there exists $1\le i \le n$ such that $s[m]$ and $s_i[m]$ are $\Delta[m]$-conjugate.
This in turn implies that for some $1\le i\le n$ sections $s[m]$ and $s_i[m]$ are $\Delta[m]$-conjugate for every $m\ge 1$.
Then sections $s$ and $s_i$ must be $\Delta[sol]$-conjugate, as $\Delta[sol]$ is the inverse limit of all groups $\Delta[m]$.
\end{proof}

When $s$ is a section of sequence \eqref{s5:ses:basic} we denote by $s^{sol}$ the composition of $s$ with the surjection $\Pi\twoheadrightarrow \Pi[sol]$, thus $s^{sol}$ is a section of sequence \eqref{s5:ses:solvable}. 

\begin{corollary}
Suppose that $s,s_1,\ldots,s_n$ are sections of sequence~\eqref{s5:ses:basic} such that sections $(s_i)_i$ cover $s$ on a set $\Omega_K\subset \V(K)$ of density one. Then there exists $1\le i \le n$ such that $s^{sol}$ and $s_i^{sol}$ are $\Delta[sol]$-conjugate. 
\end{corollary}
\begin{proof}
This follows immediately from Corollary~\ref{s5:cor:solvable_covering_property}.
\end{proof}
We now come to the proof of Theorem~\ref{s5:thm:m-truncated_section_covering}. Our first goal will be to show that sections of exact sequences~\eqref{s5:ses:m-truncated} and \eqref{s5:ses:solvable} have trivial geometric centralizers. We need a couple of lemmas.

\begin{lemma}\label{s5:lem:m-1_quotient}
Let $U$ be a normal open subgroup of $\Delta[m]$ for some $m\ge 2$ such that the quotient $\Delta[m]/U$ is abelian. Write $\Delta'$ for the preimage of $U$ along the map $\Delta\twoheadrightarrow \Delta[m]$.
Then the quotient map $\Delta'\twoheadrightarrow \Delta'[m-1]$ factors through the quotient $\Delta'\twoheadrightarrow U$.
\end{lemma}
\begin{proof}
By assumption, we have $\Delta[m]^{[1]}\subset U$ thus $\Delta^{[1]}\subset \Delta'$. Therefore \[
ker(\Delta'\twoheadrightarrow U) = \Delta'\cap \Delta^{[m]}\subset \Delta'^{[m-1]} = ker(\Delta'\twoheadrightarrow \Delta'[m-1])
\]
from which the statement follows.
\end{proof}

\begin{lemma}\label{s5:lem:commutator_lemma}
Let $G$ be a profinite group, $H\subset G$ be a closed subgroup of $G$ and $V\subset G$ be an open subgroup of $G$ such that $H^{[1]}$ is contained in $V$. Then, there exists an open subgroup $U$ of $G$ such that $H\subset U$ and $U^{[1]}\subset V$. Moreover, for every $i\ge 1$ we have 
\[
\bigcap_{H\subset U}U^{[i]} = H^{[i]},
\]
where $U$ ranges through all open subgroups of $G$ which contain $H$. Furthermore, if $G$ topologically finitely generated and $H$ is a characteristic subgroup of $G$, then in the intersection above we may restrict to open subgroups $U$ which are characteristic in $G$.
\end{lemma}
\begin{proof}
Write $f\colon G\times G \to G$ to be the map $f(x,y)=[x,y]$ for $x,y\in G$, clearly $f$ is continuous. Since $f(H\times H)\subset H^{[1]}\subset V$ there exists an open subset $U'$ of $G$ such that $H\subset U'$ and $f(U'\times U')\subset V$.
As $H$ is a closed subgroup of a profinite group, there exists an open subgroup $U$ such that $H\subset U\subset U'$. Then $f(U\times U)\subset V$ and since $V$ is a subgroup we must have $U^{[1]}\subset V$, which proves the first part of the lemma.

For the second part, fix $i\ge 1$ and let $V=V_i$ be an open subgroup of $G$ such that $H^{[i]}\subset V_i$. Applying the first part of the lemma consecutively to $H^{[i-1]},\ldots, H^{[0]}$ we obtain open subgroups $V_j$ for $0\le j\le i-1$ such that $H^{[j]}\subset V_j$ and $V_j^{[1]}\subset V_{j+1}$.
Then $U=V_0$ satisfies $H\subset U$ and $U^{[i]}\subset V$.
This implies that $\bigcap_{H\subset U}U^{[i]} \subset V$ for every open subgroup $V$ containing $H^{[i]}$. Therefore we conclude $\bigcap_{H\subset U}U^{[i]} \subset H^{[i]}$ and the inverse inclusion is obvious.

Finally, if $H$ is a closed characteristic subgroup of $G$ and $V$ is an open subgroup of $G$ containing $H$ then, as $G$ is topologically finitely generated, there exists an open characteristic subgroup $U$ of $G$ contained in $V$ such that $H\subset U$. This proves the last part of the statement.
\end{proof}

\begin{proposition}\label{s5:prop:triviality_of_centralizers}
Quasi-sections of short exact sequence~\eqref{s5:ses:m-truncated} have trivial geometric centralizers, for every $m\ge 1$. 
\end{proposition}
\begin{proof}
By restricting to an open subgroup of $G_K$ it is enough to consider the case of a section.
We will prove the statement by induction on $m$, the case $m=1$ has been discussed in Section~\ref{s:global_conjugacy_special_case}.
Let $m\ge 2$ and suppose that the geometric centralizers of sections of $m'$-truncated \'etale homotopy sequence are trivial for $m' < m$ and for all hyperbolic curves over number fields.
Let $s$ be a section of sequence~\eqref{s5:ses:m-truncated} and denote by $\delta$ an element of $\Delta[m]$ which centralizes $s$. We want to prove that $\delta=1$.
From the case $m=1$ we easily see that the image of $\delta$ in the quotient 
\[
\Delta[m]\twoheadrightarrow \Delta[1]=\Delta^{ab}
\]
is trivial. 
Take $U$ be an open characteristic subgroup of $\Delta[m]$ which contains the commutator subgroup $\Delta[m]^{[1]}$, in particular $\delta$ belongs to $U$.
Denote by $\Delta'$ the preimage of $U$ along the quotient map $\Delta\twoheadrightarrow \Delta[m]$, thus $\Delta'$ is an open characteristic subgroup of $\Delta$.
From Lemma~\ref{s5:lem:m-1_quotient} we see that the quotient $\Delta'\twoheadrightarrow \Delta'[m-1]$ factors through $U$, i.e., we have a sequence of surjections
\[
\Delta'\twoheadrightarrow U\twoheadrightarrow \Delta'[m-1].
\]
Thus, by replacing $\Delta$ with $\Delta'$ and passing to the quotient $\Delta'[m-1]$ we conclude from the inductive assumption that $\delta$ maps to the identity in the quotient $U\twoheadrightarrow \Delta'[m-1]$.
As $U$ is a quotient of $\Delta'$, it follows that $\delta$ belongs to $U^{[m-1]}$.
Since $U$ was an arbitrary open characteristic subgroup of $\Delta[m]$ containing the commutator subgroup $\Delta[m]^{[1]}$, we conclude from Lemma~\ref{s5:lem:commutator_lemma} (applied to the closed characteristic subgroup $H=\Delta[m]^{[1]}$ and $i=m-1$) that $\delta$ belongs to $H^{[m-1]}=\Delta[m]^{[m]}$ which is a trivial group.
\end{proof}

The next corollary follows immediately from Proposition~\ref{s5:prop:triviality_of_centralizers}.

\begin{corollary}\label{s5:cor:solvable_centralizers}
Quasi-sections of short exact sequence~\eqref{s5:ses:solvable} have trivial geometric centralizers.
\end{corollary}
\begin{proof}
If $\delta$ is an element of $\Delta[sol]$ which centralizes a quasi-section of sequence~\eqref{s5:ses:solvable} then by Proposition~\ref{s5:prop:triviality_of_centralizers} we obtain that the image of $\delta$ is trivial in the quotient $\Delta[sol]\twoheadrightarrow \Delta[m]$ for every $m\ge 1$, thus $\delta=1$. 
\end{proof}

We are now able to finish the proof of Theorem~\ref{s5:thm:m-truncated_section_covering}.

\begin{proof}[Proof of Theorem \ref{s5:thm:m-truncated_section_covering}]
We prove the statement by induction on $m$, the case $m=1$ is equivalent to the finite covering property of $\Delta[1] = \Delta^{ab}$ and has been proved in Section~\ref{s:covering_property}.

Fix $m\ge 2$ and assume that the statement holds for all integers smaller than $m$ and all hyperbolic curves over number fields.
Let $s$ and $s_1,\ldots, s_n$ be sections of sequence~\eqref{s5:ses:m-truncated} such that sections $(s_i)_i$ cover section $s$ on a set of valuations $\Omega_K\subset \V(K)$ of density one. Write $\Omega$ for the preimage of $\Omega_K$ along the surjection $\V\twoheadrightarrow \V(K)$.
Let $U$ be an open characteristic subgroup of $\Delta[m]$. Define the group $\Pi(U)$ to be the neighbourhood of $s$ determined by $U$, hence we have a short exact sequence
\[
1\to U\to \Pi(U)\to G_K\to 1.
\]
Our goal is to find an integer $1\le i \le n$ such that some $\Delta[m]$-conjugate section $t_i$ of $s_i$ satisfies $t_i(G_K)\subset \Pi(U)$. Once we know this for all open characteristic subgroups $U$ of $\Delta[m]$, it will follow from Lemma~\ref{s3:lem:conjugacy_criterion} that for some $1\le i \le n$ sections $s_i$ and $s$ are $\Delta[m]$-conjugate.

In order to find such a section $s_i$ we will introduce an auxiliary subgroup of $\Delta[m]$ which will allow us to use the inductive assumption on $m$.
Define a subgroup $V$ of $\Delta[m]$ generated by $U$ and by the commutator subgroup $\Delta[m]^{[1]}$. Note that $V$ is an open characteristic subgroup of $\Delta[m]$ which contains $U$ and the quotient $\Delta[m]/V$ is abelian. Moreover, the subgroup $V^{[m-1]}$ is contained in $U$. Indeed, the image of $V$ in the quotient 
\[
\Delta[m]\twoheadrightarrow \Delta[m]/U
\]
is generated by $\Delta[m]^{[1]}$, thus the image of $V^{[m-1]}$ is trivial. 
Denote by $\Pi(V)$ the neighbourhood of section $s$ determined by $V$. Thus $\Pi(U)\subset \Pi(V)\subset \Pi$ and we have a short exact sequence
\begin{equation}\label{s5:ses:auxiliary_v_cover}
1\to  V\to \Pi(V)\to G_K\to 1.
\end{equation}
Let $\Delta'$ be the preimage of $V$ along the quotient $\Delta\twoheadrightarrow \Delta[m]$, thus $\Delta'$ is an open characteristic subgroup of $\Delta$.
By Lemma~\ref{s5:lem:m-1_quotient} we have a sequence of quotient maps 
\[
\Delta'\twoheadrightarrow V\twoheadrightarrow\Delta'[m-1].
\]
Applying the quotient $V\twoheadrightarrow \Delta'[m-1]$ to sequence~\ref{s5:ses:auxiliary_v_cover} we also have an exact sequence
\begin{equation}\label{s5:ses:m-1_cover}
1\to \Delta'[m-1]\to \Pi'(m-1)\to G_K\to 1.
\end{equation}
and we write $\varphi\colon\Pi(V)\twoheadrightarrow \Pi'(m-1)$ for the natural surjection.

Let $k$ be the index of $V$ in $\Delta[m]$. According to the discussion in Section~\ref{s:quasi-sections}, for every $1\le i\le n$ section $D_i = s_i(G_K)$ of sequence~\eqref{s5:ses:m-truncated} splits into $k$ quasi-sections $D'_{ij}$ of sequence~\eqref{s5:ses:auxiliary_v_cover}.
Here we recall that $D_{ij} = \delta_j D_i \delta_j^{-1}$, for $1\le j\le k$, are various $\Delta[m]$-conjugates of $D_i$ where $\delta_j$ are representatives of right cosets of $V$ in $\Delta[m]$ and $D'_{ij} = \Pi(V)\cap D_{ij}$ is an open subgroup of $D_{ij}$.
Moreover, by Lemma~\ref{s3:lem:stability_of_coverings} quasi-sections $(D'_{ij})_{ij}$ cover section $s$ on $\Omega$; here we consider $s$ as a section of sequence~\eqref{s5:ses:auxiliary_v_cover}.

Let $L$ be a finite field extension of $K$ such that all quasi-sections $D'_{ij}$ become sections over $G_L$, write $s_{ij}$ for a section of the short exact sequence
\begin{equation}\label{s5:ses:auxiliary_v_cover_restriction}
1\to V\to \Pi(V)_L\to G_L\to 1
\end{equation}
determined by $D'_{ij}$, here $\Pi(V)_L$ is the preimage of $G_L$ along the surjection $\Pi(V)\twoheadrightarrow G_K$. Similarly we have a short exact sequence
\begin{equation}\label{s5:ses:m-1_cover_restriction}
1\to \Delta'[m-1]\to \Pi'(m-1)_L\to G_L\to 1,
\end{equation}
obtained by restricting sequence~\eqref{s5:ses:m-1_cover} to $G_L$.

By construction sections $(s_{ij})_{ij}$ of sequence~\eqref{s5:ses:auxiliary_v_cover_restriction} cover the restricted section $s|_{G_L}$ on a set $\Omega_L\subset \V(L)$ of density one, where $\Omega_L$ is the image of $\Omega$ through the surjection $\V\twoheadrightarrow \V(K)$.
Composing sections $s|_{G_L}$ and $s_{ij}$ with $\varphi$ and using the inductive assumption we conclude that there exist integers $i,j$ such that sections $\varphi\circ s|_{G_L}$ and $\varphi\circ s_{ij}$ of sequence~\eqref{s5:ses:m-1_cover_restriction} are $\Delta'[m-1]$-conjugate.
Modifying the representative $\delta_j$ by an element of $V$, which changes $D_{ij}$ into some $V$-conjugate subgroup, we may assume that sections $\varphi\circ s|_{G_L}$ and $\varphi\circ s_{ij}$ of sequence~\eqref{s5:ses:m-1_cover_restriction} are in fact equal.
Write $t$ for a section of sequence~\eqref{s5:ses:m-truncated} determined by $D_{ij}$, we are going to prove that $t(G_K)$ is contained in $\Pi(U)$, which will finish the proof.

Let us first show that $t(G_K)$ is contained in $\Pi(V)$. As $m\ge 2$, we may compose the quotient $V\twoheadrightarrow \Delta'[m-1]$ with $\Delta'[m-1]\twoheadrightarrow \Delta'[1] = (\Delta')^{ab}$, write $\psi: \Pi(V)\twoheadrightarrow \Pi'(1)$ for the induced map. Since the map $\psi$ factors through $\varphi$ we see that compositions $\psi\circ s|_{G_L}$ and $\psi\circ s_{ij}$ are equal.
On the other hand, consider the abelianized short exact sequence
\begin{equation}\label{s5:ses_abelianized}
1\to \Delta[1]\to \Pi(1)\to G_K\to 1,
\end{equation}
and its restriction to $G_L$
\begin{equation}\label{s5:ses_abelianized_restricted}
1\to \Delta[1]\to \Pi(1)_L\to G_L\to 1.
\end{equation}
Note that the natural map $\Delta'[1]\to \Delta[1]$ makes the following diagram commutative
\[
\begin{tikzcd}
V \arrow[r, two heads]\arrow[d, hook] & \Delta'[1]\arrow[d] \\
\Delta[m]\arrow[r, two heads]   & \Delta[1].
\end{tikzcd}
\]   
Thus we may deduce that sections $s^{ab}|_{G_L}$ and $t^{ab}|_{G_L}$ of sequence \eqref{s5:ses_abelianized_restricted} are equal. By Proposition~\ref{s5:prop:triviality_of_centralizers} applied to $m=1$ and Lemma~\ref{s3:lem:equality_on_open_subgroup} this implies that in fact $s^{ab} = t^{ab}$, in other words sections $s$ and $t$ differ by a cocycle with values in $\Delta[m]^{[1]}$.
Since the quotient $\Delta[m]/V$ is abelian and $\Pi(V)$ is a neighbourhood of $s$ this implies that $t(G_K)$ is contained in $\Pi(V)$. 

Thus $t$ is in fact a section of sequence \eqref{s5:ses:auxiliary_v_cover}, hence $s_{ij} = t|_{G_L}$. Applying the map $\varphi$ we deduce
\[
 \varphi\circ s|_{G_L} = \varphi\circ s_{ij} = \varphi\circ t|_{G_L}.
\]
Similarly, by using Proposition~\ref{s5:prop:triviality_of_centralizers} for $m-1$ and Lemma~\ref{s3:lem:equality_on_open_subgroup} we infer that in fact $\varphi\circ s = \varphi \circ t$.
Moreover, from the inclusion $V^{[m-1]}\subset U$ we see that the quotient $V\twoheadrightarrow V/U$ factors through the quotient $V\twoheadrightarrow \Delta'[m-1]$.
Therefore sections $s$ and $t$ differ by a cocycle with values in $U$. As $\Pi(U)$ is a neighbourhood of section $s$ we conclude that $t(G_K)$ is contained in $\Pi(U)$. This concludes the proof.
\end{proof}

\section{Global conjugacy of sections}\label{s:global_conjugacy}
In this section we prove our main theorem concerning the finite covering property of the \'etale homotopy short exact sequence of a hyperbolic curve over a number field. To achieve this we will use the maximal pro-solvable case, proved in Section~\ref{s:pro-solvable}, for all open characteristic subgroups of the geometric fundamental group. As in the previous section, we consider the short exact sequence
\begin{equation}\label{s6:ses:basic}\tag{\'etale}
1\to \Delta\to \Pi\to G_K\to 1
\end{equation}
associated to a hyperbolic curve $X$ over a number field $K$.

First, recall the following well-known lemma.
\begin{lemma}\label{s6:lem:general_torsion-freeness}
Let $G$ be a profinite group. Suppose that for every open subgroup $U$ of $G$ the abelianization $U^{ab}$ is torsion-free. Then $G$ is torsion-free.
\end{lemma}
\begin{proof}
See~\cite[Lem.~1.5]{tamagawa1997grothendieck}.
\end{proof}

The above lemma is used to obtain the following fact.

\begin{lemma}\label{s6:lem:torsion_freeness}
Let $H$ be an open normal subgroup of $\Delta$. Then the group $\Delta/H^{[sol]}$ is torsion-free.
\end{lemma}
\begin{proof}
Let $U$ be an open subgroup of $G=\Delta/H^{[sol]}$, by Lemma~\ref{s6:lem:general_torsion-freeness} it is enough to check that $U^{ab}$ is torsion-free. Write $V$ for an open subgroup of $\Delta$ obtained as the preimage of $U$ along the surjection $\Delta\twoheadrightarrow G$, thus we have $H^{[sol]}\subset V $ and $ U=V/H^{[sol]}$. Observe that we have an equality 
\[
[H^{[sol]},H^{[sol]}]=H^{[sol]}.
\]
This implies that $U^{ab}$ is isomorphic to $V^{ab}$ which is torsion-free.
\end{proof}

Let $H$ be an open characteristic subgroup of $\Delta$, after passing to the quotient $\Delta\twoheadrightarrow \Delta/H^{[sol]}$ we obtain a short exact sequence 
\begin{equation}\label{s6:ses:open_solvable}
1\to \Delta/H^{[sol]}\to \Pi/H^{[sol]}\to G_K\to 1.
\end{equation}

\begin{lemma}\label{s6:lem:triviality_of_centralizers}
Every quasi-section of short exact sequence~\eqref{s6:ses:open_solvable} has trivial geometric centralizer.
\end{lemma}
\begin{proof}
Let $D$ be a quasi-section of sequence~\eqref{s6:ses:open_solvable} and let $\delta$ be an element of $\Delta/H^{[sol]}$ centralizing $D$. Note that we have an exact sequence
\[
1\to H[sol]\to \Delta/H^{[sol]}\to \Delta/H\to 1,
\] 
and since the group $\Delta/H$ is finite there exists a natural number $n\ge 1$ such that $\delta^n$ belongs to $H[sol]$.
Clearly $\delta^n$ also centralizes $D$ thus replacing $D$ by some open subgroup if necessary we may apply Corollary~\ref{s5:cor:solvable_centralizers} and conclude that $\delta^{n}=1$.
Finally, Lemma~\ref{s6:lem:torsion_freeness} implies that $\delta=1$.
\end{proof}

From our results in Section~\ref{s:quasi-sections} we obtain the following corollary.

\begin{corollary}
Let $\bar{s}$ and $\bar{t}$ be two sections of sequence~\eqref{s6:ses:open_solvable}. Suppose that $\bar{s}$ and $\bar{t}$ coincide on an open subgroup of $G_K$. Then, we have $\bar{s}=\bar{t}$.
\end{corollary}
\begin{proof}
This follows immediately from Lemma~\ref{s3:lem:equality_on_open_subgroup} and Lemma~\ref{s6:lem:triviality_of_centralizers}.
\end{proof}

We are now ready to prove the main theorem of this article.

\begin{theorem}
Short exact sequence~\eqref{s6:ses:basic} has finite covering property.
\end{theorem}
\begin{proof}
Let $s$ and $s_1,\ldots, s_n$ be sections of sequence~\eqref{s6:ses:basic} for some $n\ge 1$ such that sections $(s_i)_i$ cover section $s$ on a set of valuations $\Omega_K\subset \V(K)$ of density one.
Let $\Delta'$ be an open characteristic subgroup of $\Delta$ and define $\Pi'$ to be the neighbourhood of section $s$ determined by $\Delta'$. Thus $\Pi'$ is an open subgroup of $\Pi$ and we have a short exact sequence
\[
1\to \Delta'\to \Pi'\to G_K\to 1.
\]
By Lemma~\ref{s3:lem:conjugacy_criterion} it is enough to show that for some $1\le i \le n$ there exists a $\Delta$-conjugate section $t_i$ of section $s_i$ such that $t_i(G_K)\subset \Pi'$.

We proceed similarly as in the proofs of Theorems~\ref{s4:thm:global_conjugacy} and ~\ref{s5:thm:m-truncated_section_covering}.
Write $D=s(G_K)$ and $D_i=s_i(G_K)$ for $1\le i\le n$, from the discussion in Section~\ref{s:quasi-sections} we have a splitting of the $\Delta$-torsor $(\Delta, D_i)$ as follows
\[
\Pi' \cap (\Delta, D_i) = \bigsqcup_{j=1}^{k}(\Delta', D'_{ij}).
\]
Here $D_{ij}=\delta_j D_i \delta_j^{-1}$ for a system of representatives $\delta_j$ of right cosets of $\Delta'$ in $\Delta$ and $D'_{ij}= \Pi' \cap D_{ij}$.

Write $\Omega$ for the preimage of $\Omega_K$ along the surjection $\V\twoheadrightarrow \V(K)$.
Take a finite field extension $L/K$ such that after restricting to $G_L$ all quasi-sections $D'_{ij}$ become sections and write $s_{ij}$ for the section determined by $D_{ij}$. Denote by $\Omega_L$ the image of $\Omega$ through the surjection $\V\twoheadrightarrow \V(L)$, it is a set of density one.
Then $s|_{G_L}$ and $s_{ij}$ are sections of the short exact sequence 
\[
1 \to \Delta' \to \Pi'_L \to G_L\to 1,
\]
where $\Pi'_L$ is the preimage of $G_L$ along the surjection $\Pi'\twoheadrightarrow G_K$. By Lemma~\ref{s3:lem:stability_of_coverings} we see that sections $(s_{ij})_{ij}$ cover section $s|_{G_L}$ on $\Omega_L$. Applying the quotient $\Delta'\twoheadrightarrow \Delta'[sol]$ we have a short exact sequence
\begin{equation}\label{s6:eq:ses_solvable_rest}
1\to \Delta'[sol]\to \Pi'_L(sol)\to G_L\to 1,
\end{equation}
together with sections $s|_{G_L}^{sol}$ and $s^{sol}_{ij}$ obtained by composing sections $s|_{G_L}$ and $s_{ij}$ with the quotient map $\Pi'_L\twoheadrightarrow \Pi'_L(sol)$.
Thus, by applying Corollary~\ref{s5:cor:solvable_covering_property} to sequence~\eqref{s6:eq:ses_solvable_rest} we conclude that there exist integers $i,j$ such that sections $s|_{G_L}^{sol}$ and $s^{sol}_{ij}$ are $\Delta'[sol]$-conjugate.
Modifying $\delta_j$ by an element of $\Delta'$, which changes $D_{ij}$ to its $\Delta'$-conjugate, we may assume that we have in fact the equality $s|_{G_L}^{sol} = s^{sol}_{ij}$.
Write $t$ for a section of sequence~\eqref{s6:ses:basic} determined by $D_{ij}$, thus $s_{ij}=t|_{G_L}$. We are going to show that $t(G_K)$ is contained in $\Pi'$, which will finish the proof.

Consider the following short exact sequence
\begin{equation}\label{s6:eq:last_sequence}
1\to \Delta/\Delta'^{[sol]}\to \Pi/\Delta'^{[sol]}\to G_K\to 1.
\end{equation}
Write $\bar{s}$ and $\bar{t}$ for the compositions of sections $s$ and $t$ with the quotient map $\Pi\twoheadrightarrow \Pi/\Delta'^{[sol]}$, hence $\bar{s}$ and $\bar{t}$ are both sections of sequence~\eqref{s6:eq:last_sequence}.
It is enough to show that $\bar{s}=\bar{t}$ since this implies that $s$ and $t$ differ by a cocycle with values in $\Delta'^{[sol]}\subset \Delta'$ from which it follows that $t(G_K)$ is contained in $\Pi'$.
Moreover, by applying Lemma~\ref{s6:lem:triviality_of_centralizers} to sequence~\eqref{s6:eq:last_sequence} we see that to prove the equality $\bar{s}=\bar{t}$ we may restrict both sections to an open subgroup $G_L$ of $G_K$.
By definition of $G_L$, the images $s(G_L)$ and $t(G_L)$ are contained in $\Pi'$, hence $\bar{s}|_{G_L}$ and $\bar{t}|_{G_L}$ are in fact sections of sequence~\eqref{s6:eq:ses_solvable_rest}. Therefore we have $\bar{s}|_{G_L}=s^{sol}|_{G_L}$ and $\bar{t}|_{G_L}=s^{sol}_{ij}$, from which we conclude $\bar{s}|_{G_L}=\bar{t}|_{G_L}$.
\end{proof}

\section*{Acknowledgements}
\indent The author would like to express his gratitude to Yuichiro Hoshi for listening to the first version of the argument and to Shinichi Mochizuki for his comment which led to the statement of Lemma~\ref{s6:lem:torsion_freeness}.

\bibliographystyle{abbrv}
\bibliography{bibliography}

\end{document}